\documentclass[12 pt]{amsart}

\usepackage{hyperref}
\usepackage{etex}
\usepackage[shortlabels]{enumitem}
\usepackage{amsmath}
\usepackage{amsxtra}
\usepackage{amscd}
\usepackage{amsthm}
\usepackage{amsfonts}
\usepackage{amssymb}
\usepackage{eucal}
\usepackage[all]{xy}
\usepackage{graphicx}
\usepackage{tikz-cd}
\usepackage{mathrsfs}
\usepackage{subfiles}
\usepackage{mathpazo}
\usepackage[colorinlistoftodos, textsize=tiny]{todonotes}
\usepackage{morefloats}
\usepackage{pdfpages}
\usepackage{thm-restate}
\usepackage[utf8]{inputenc}
\usepackage{epigraph}
\usepackage{csquotes}
\usepackage[margin=1.0in]{geometry}
\usepackage{stmaryrd}

\graphicspath{ {images/} }

\RequirePackage{color}
\definecolor{myred}{rgb}{0.75,0,0}
\definecolor{mygreen}{rgb}{0,0.5,0}
\definecolor{myblue}{rgb}{0,0,0.65}

\usepackage{hyperref}
\hypersetup{citecolor=blue}
\usepackage{tikz}
\usetikzlibrary{matrix,arrows,decorations.pathmorphing}

\theoremstyle{plain}
\newtheorem{theorem}{Theorem}[section]
\newtheorem{proposition}[theorem]{Proposition}
\newtheorem{lemma}[theorem]{Lemma}
\newtheorem{corollary}[theorem]{Corollary}
\theoremstyle{definition}
\newtheorem{definition}[theorem]{Definition}
\newtheorem{remark}[theorem]{Remark}
\newtheorem{example}[theorem]{Example}

\newtheorem{question}[theorem]{Question}

\theoremstyle{remark}
\newtheorem{notation}[theorem]{Notation}
\numberwithin{equation}{section}
  
\newcommand\nc{\newcommand}
\nc\on{\operatorname}
\nc\renc{\renewcommand}

\newcommand\bc{{\mathbb C}}
\newcommand\br{{\mathbb R}}

\newcommand\bp{{\mathbb P}}

\newcommand\bz{{\mathbb Z}}
\newcommand\ba{{\mathbb A}}
\newcommand\bg{{\mathbb G}}

\newcommand\scc{\mathscr C}

\newcommand\scj{\mathscr J}

\newcommand\scl{\mathscr L}

\newcommand\sco{\mathscr O}

\newcommand\scx{\mathscr X}
\newcommand\scy{\mathscr Y}

\newcommand \ra{\rightarrow}
\newcommand \xra{\xrightarrow}
\DeclareMathOperator\spec{\text{Spec}}

\newcommand*{\shom}{\mathscr{H}\kern -.5pt om}
\newcommand*{\stor}{\mathscr{T}\kern -.5pt or}
\newcommand*{\sext}{\mathscr{E}\kern -.5pt xt}
\newcommand \mg{{\mathscr M_g}}
\newcommand \mgc{{\mathscr M_g^c}}
\newcommand \hg{{\mathscr H_g}}
\newcommand \ag{{\mathscr A_g}}

\makeatletter
\newcommand{\customlabel}[2]{\protected@write \@auxout {}{\string \newlabel {#1}{{#2}{\thepage}{#2}{#1}{}} }\hypertarget{#1}{#2}}

\newcommand\hgb[1]{\mathscr H_{g,#1}} 
\newcommand\mgb[1]{\mathscr M_{g,#1}} 
\newcommand\agb[1]{\mathscr A_{g,#1}} 
\DeclareMathOperator\id{id}

\DeclareMathOperator\coker{coker}

\DeclareMathOperator\pic{Pic}

\DeclareMathOperator\spn{Span}
\DeclareMathOperator\im{im}

\DeclareMathOperator\sym{Sym}
\DeclareMathOperator\xsym{Sym_2}
\DeclareMathOperator\pgl{PGL}

\DeclareMathOperator\chr{\text{char}}

\newcommand\ul{\underline}
\newcommand\ol{\overline}

\DeclareMathOperator\isom{isom}

\DeclareMathOperator\aut{Aut}
\DeclareMathOperator\gl{GL}

\DeclareFontFamily{U}{wncy}{}
\DeclareFontShape{U}{wncy}{m}{n}{<->wncyr10}{}
\DeclareSymbolFont{mcy}{U}{wncy}{m}{n}
\DeclareMathSymbol{\Sha}{\mathord}{mcy}{"58}

\def\listtodoname{List of Todos}
\def\listoftodos{\@starttoc{tdo}\listtodoname}

\title{The Torelli map restricted to the hyperelliptic locus}
\author{Aaron Landesman}

\usepackage{microtype}
\begin{document}

\begin{abstract}
	Let $g \geq 2$ and let the Torelli map denote the map sending a genus $g$ curve to its principally polarized Jacobian.
	We show that the restriction of the Torelli map to the hyperelliptic locus is an immersion in characteristic not $2$.
	In characteristic $2$, we show the Torelli map restricted to the hyperelliptic locus 
	fails to be an immersion because it is generically inseparable;
	moreover, the induced map on tangent spaces has kernel of dimension $g-2$ at every point.
\end{abstract}

\maketitle

\section{Introduction}
Let $\hg$ over $\spec \bz$ denote the moduli stack of smooth hyperelliptic curves of genus $g$, $\mg$ over $\spec \bz$ denote the moduli stack of
smooth curves of genus $g$, and $\ag$ over $\spec \bz$ denote the moduli stack of principally polarized abelian varieties of dimension $g$.
Throughout, for $S$ a scheme, by a {\em curve of genus $g$ over $S$}, we mean a smooth proper morphism of schemes $f: C \to S$
of relative dimension $1$ whose fibers are geometrically connected $1$-dimensional schemes of arithmetic genus $g$.
For $R$ a ring, we use $\hgb R, \mgb R, \agb R$ to denote the base changes of $\hg, \mg$, and $\ag$ over $\spec \bz$ along $\spec R \ra \spec \bz$.
Throughout, we assume $g \geq 2$.

The main goal of this paper is to understand whether the composition
$\hg \xrightarrow{\iota_g} \mg \xra{\tau_g} \ag$ is an immersion, (i.e., a locally closed immersion,) for $\tau_g$ the Torelli map sending a curve to its principally polarized Jacobian.
Let $\phi_g: \hg \ra \ag$ denote this composition.
We use $\tau_{g,R}$ and $\phi_{g,R}$ for the base changes of $\tau_g$ and $\phi_g$ along a map $\spec R \ra \spec \bz$.
To start, we recall the classical characterization of when $\tau_g$ is injective on tangent vectors.
This follows from \protect{\cite[Theorems 2.6 and 2.7]{oortS:local-torelli-problem}}
together with the converse of
\protect{\cite[Theorem 2.7]{oortS:local-torelli-problem}}, which is easy to verify directly, see \autoref{lemma:dual-tangent-space-description}.
Also, see \cite[Theorem 4.3]{landesman:the-infinitesimal-torelli-problem}.

\begin{theorem}[\protect{\cite[Theorems 2.6 and 2.7]{oortS:local-torelli-problem}}]
	\label{theorem:torelli}
	Let $k$ be a field. For $g \geq 3$, and $[C] \in \mgb k$ a field-valued point, the Torelli map $\tau_{g,k}: \mgb k \ra \agb k$ is injective on tangent vectors at $[C]$
	if and only if $[C] \in \mgb k - \hgb k \subset \mgb k$.
	When $g = 2$, the map $\tau_{g,k}$ is injective on tangent vectors at all points $[C] \in \mgb k$.
\end{theorem}

Moreover, away from characteristic $2$,
the precise scheme theoretic fiber of $\tau_g$ over geometric points corresponding to hyperelliptic 
Jacobians is computed in \cite[p. 7]{ricolfi:hilbert-scheme-hyperelliptic}.
As a consequence, \autoref{theorem:torelli} shows that $\tau_g$ is not even a monomorphism at points of $\hg$ when $g \geq 3$.
It is therefore natural to ask whether the restriction $\tau_g|_{\hg} = \phi_g : \hg \ra \ag$ is a monomorphism.
Our main theorem answers this question.
We say a morphism of algebraic stacks is a {\em radimmersion} if it factors as the composition
of a finite radicial morphism and an open immersion, see \autoref{definition:radimmersion}.
\begin{theorem}
	\label{theorem:hyperelliptic}
	For $g \geq 2$, the map $\phi_g: \hg \ra \ag$ over $\spec \bz$ is a radimmersion.
	Additionally, $\phi_2$ is an immersion and $\phi_{g, \bz[1/2]}: \hgb {\bz[1/2]} \ra \agb {\bz[1/2]}$ is an immersion.
	However, when $g > 2$, for $k$ a field of characteristic $2$, $\phi_{g, k}$
	is not an immersion; instead, $\phi_{g, k}$ is generically inseparable and the induced
	map on tangent spaces at any geometric point of $\hgb k$ has kernel of dimension $g-2$. 
\end{theorem}

We carry out the proof of \autoref{theorem:hyperelliptic} at the end of \autoref{section:proof-hyperelliptic}.
To paraphrase the statement, \autoref{theorem:torelli} says, loosely speaking, that there are many tangent vectors to a hyperelliptic point in $\mg$ that are killed under $\tau_g$. 
We wish to understand whether those tangent vectors which are killed can lie inside $\hg$, or whether they correspond to deformations to non-hyperelliptic curves.
The answer, provided by \autoref{theorem:hyperelliptic}, is that they do all correspond to deformations to non-hyperelliptic curves when the characteristic is not $2$,
but this fails quite badly in characteristic $2$.

\begin{remark}
	\label{remark:}
	The statement that $\phi_{g,k}$ is an immersion for a field $k$ of characteristic not $2$ appears in \cite[p. 176]{oortS:local-torelli-problem},
	though some details are omitted. Our guess is that the authors verified the map $\phi_{g,k}$ is injective on tangent spaces via explicitly calculating the Kodaira Spencer map sending a differential
	to a corresponding deformation, in order to understand the image of the map $T_{[C]} \hgb k \ra T_{[C]} \mgb k \simeq H^1(C, T_C)$ on tangent spaces at a point $[C] \in \hgb k$.
	However, no indication is given there as to how to use this to show the map of stacks is an immersion.
	In this article we take a different approach, which also applies in characteristic $2$;
	we do not see how to carry out the approach of \cite{oortS:local-torelli-problem} in characteristic $2$.
\end{remark}
\begin{remark}
	\label{remark:}
	Our initial interest in this problem was motivated by the computation of the essential dimension of level structure covers of $\hg$.
	Using \autoref{theorem:hyperelliptic}, it is shown in \cite[Example 2.3.6]{farbKW:modular-functions-and-resolvent-problems}
	that for $g \geq 2$, the cover of $\hg$ parameterizing pairs of a hyperelliptic curve and an $n$-torsion point of the Jacobian of that curve
	(over a field of characteristic $0$)
	has essential dimension $2g-1$ when $2 \nmid n$ but only has essential dimension $g + 1$ when $n = 2$.
\end{remark}

There are two main components to the proof of \autoref{theorem:hyperelliptic}.
The first component of the proof is to describe the map $\phi_g$ induces on tangent vectors.
This is done by analyzing the deformation theory of hyperelliptic curves, which is possible by means of their relatively simple equations.
The key tool to analyzing the induced map on tangent spaces is \autoref{lemma:dual-tangent-space-description}, which relies on a nonstandard definition of $\hg$
given in \autoref{subsection:hg}.
The second component of the proof is to very that $\phi_g$ is a radimmersion.
This will imply $\phi_g$ is an immersion when it is a monomorphism, i.e., away from characteristic $2$.
For checking $\phi_g$ is a radimmersion, we use a valuative criterion, 
which roughly says that a map of stacks $f: X \ra Y$ is a radimmersion
when, given a map from the spectrum a discrete valuation ring to $Y$ with its two points factoring through $X$, the map from the 
spectrum of the discrete valuation ring factors uniquely through $X$.
We verify this valuative criterion using that $\phi_g$ factors as the composition of an immersion into the 
moduli stack of compact type curves $\hg \ra \mgc$, and a proper ``compactified Torelli map'' $\mgc \ra \ag$.

The outline of the paper is as follows:
In \autoref{section:background} we recall background on the infinitesimal Torelli theorem and
the moduli stack of hyperelliptic curves.
In \autoref{section:general-setup} we describe the map $\phi_g$ induces on tangent spaces.
We then compute this map on tangent spaces when the characteristic is not $2$ in \autoref{section:char-not-2}
and when the characteristic is $2$ in \autoref{section:char-2}.
After some preliminaries on the compactified Torelli map, we prove \autoref{theorem:hyperelliptic} in \autoref{section:proof-hyperelliptic}.
Finally, in \autoref{section:valuative}, we prove a valuative criterion for immersions of stacks, which is used in the proof
of \autoref{theorem:hyperelliptic}.

\subsection{Acknowledgements}
We would like to thank Bogdan Zavyalov for many helpful discussions.
We thank Rachel Pries for asking the question which led to the genesis of this article, and for much help understanding hyperelliptic curves
in characteristic $2$.
We thank an anonymous referee for numerous helpful comments and suggestions.
We also thank Benson Farb, Mark Kisin, and Jesse Wolfson for helpful correspondence and encouraging us to write this paper.
We thank Brian Conrad for the idea to use the valuative criterion for immersions.
We thank Valery Alexeev, Sean Cotner, Johan de Jong, Martin Olsson, Frans Oort, Bjorn Poonen, Will Sawin, and Ravi Vakil for further helpful discussions.
This material is based upon work supported by the National Science Foundation Graduate Research Fellowship Program under Grant No. DGE-1656518.

\section{Background}
\label{section:background}
In this section, we review relevant background notation we will need from \autoref{theorem:torelli} in \autoref{subsection:inputs} 
and also
a nonstandard construction of the moduli stack of hyperelliptic 
curves which will be crucial to the ensuing proof.
We define the stack $\hg$ of genus $g$ hyperelliptic curves in
\autoref{subsection:hg}.
We show $\hg$ is a smooth algebraic stack in \autoref{subsection:algebraic-stack}.
Finally, in 
\autoref{subsection:closed-immersion},
we show $\hg$ has a closed immersion into $\mg$.

\subsection{Key inputs in the proof of the infinitesimal Torelli theorem}
\label{subsection:inputs}

We next review the key inputs in the proof of \autoref{theorem:torelli}, as we will rely on understanding
explicitly the map on tangent spaces associated to $\tau_g: \mg \ra \ag$ in our ensuing analysis of the map $\hg \ra \ag$.

The statement regarding injectivity on geometric points is the classical Torelli theorem 
\cite[Chapter VII, Theorem 12.1(a)]{CornellS:Arithmetic},
see also the original proof by Torelli \cite{torelli:sulle-varieta-di-jacobi} and Andreotti's beautiful proof \cite{andreotti:on-a-theorem-of-torelli}.
Thus, we just address the statement on tangent vectors.
	This too is classical, and boils down to Noether's theorem, regarding the map 
	\eqref{equation:torelli-dual-map} below,
	though is perhaps
	less well known.

	Let $k$ be a field and let $[C] \in \mgb k$ be a field valued point of $\mgb k$ corresponding to the curve $C$.
	We'd like to understand whether the map
	\begin{align*}
		T_{[C]} \mgb k \ra T_{[\tau_{g,k}(C)]} \agb k
	\end{align*}
	is injective, for $T_{[C]} \mgb k$ denoting the tangent space to $\mgb k$ at $[C]$.
	By deformation theory, $T_{[C]} \mgb k \simeq H^1(C, T_C)$.
	For $V$ a vector space, define $\xsym V$ as the kernel of $V^{\otimes 2} \ra \wedge^2 V$, where $\wedge^2 V = V^{\otimes 2}/ \spn(v\otimes v : v \in V)$.
	(Note that $\xsym V \simeq \sym^2 V$ in characteristic $p \neq 2$, but differs in characteristic $2$. Here $\sym^2 V$ denotes the natural quotient of $V \otimes V$ by the span of $v \otimes w - w \otimes v$ for $v, w \in V$.)
	Further,
	\begin{align*}
		T_{[\tau_{g}(C)]}\agb k &= \ker(H^1(\tau_g(C), \sco_{\tau_g(C)})^{\otimes 2} \ra \wedge^2 H^1(\tau_g(C), \sco_{\tau_g(C)}) ) \\
	&\simeq \xsym H^1(C, \sco_C).
	\end{align*}
	as described in the proof of \cite[Theorem 2.6]{oortS:local-torelli-problem}.
	Also see \cite[Theorem 3.3.11(iii)]{sernesi:deformations-of-algebraic-schemes} for this identification.

	Therefore, we wish to understand whether the natural map
	\begin{align}
		\label{equation:torelli-tangent-map}
		H^1(C, T_C) \ra \xsym H^1(C, \sco_C),
	\end{align}
	induced by $T_{[C]}\tau_{g,k}: T_{[C]} \mgb k \ra T_{[\tau_g(C)]}\agb k$, is injective.

	Applying Serre duality,
	since $H^1(C, T_C)$ is dual to $H^0(C, \omega_C^{\otimes 2})$,	
	it is equivalent to understand surjectivity of the corresponding
	map
	\begin{align}
		\label{equation:torelli-dual-map}
		H^0(C, \omega_C^{\otimes 2}) \leftarrow \sym^2 H^0(C, \omega_C).
	\end{align}
	This duality, valid even in characteristic $2$, uses that, for $V$ a finite dimensional vector space, 
	$\xsym V^\vee$ can be viewed as the second graded piece of the algebra $\sym_\bullet V^\vee$ with its divided power structure that is naturally dual to $\sym^\bullet V$.
	The map \eqref{equation:torelli-dual-map} is explicitly the map given by multiplying two sections, 
	see
	\cite[Theorem 4.3]{landesman:the-infinitesimal-torelli-problem} and \cite[Theorem 2.6]{oortS:local-torelli-problem}.
	By Noether's theorem \cite[Theorem 2.10]{saint-donat:on-petris-analysis-of-the-linear-system-of-quadrics}
	the map
	\eqref{equation:torelli-dual-map}
	is surjective when $C$ is not hyperelliptic
	and fails to be surjective when $C$ is hyperelliptic.
	See \autoref{remark:hyperelliptic-image-char-0} and \autoref{remark:hyperelliptic-image-char-2} for an explicit description of the image of \eqref{equation:torelli-dual-map}
	in the
	hyperelliptic case.

\subsection{Definition of $\hg$}
\label{subsection:hg}
There are several different definitions of $\hg$, the moduli stack of hyperelliptic curves of genus $g$, in the literature. For the purposes of this paper, we will be especially concerned with the more delicate case when $2$ is not invertible
on the base, so let us now expand a bit on the definition of $\hg$ over $\spec \bz$ we employ.
We will essentially define $\hg$ as the Hurwitz stack of degree $2$ covers of a genus $0$ curve.
We assume $g \geq 2$.
For the next definition,
recall that a map $\phi : X \to Y$ is {\em locally free of degree $d$} if $\phi_* \sco_X$ is a locally free rank $d$ sheaf on $Y$, 
or equivalently $\phi$ is a degree $2$ finite map which is flat and of finite presentation
\cite[\href{https://stacks.math.columbia.edu/tag/02KB}{Tag 02KB}]{stacks-project}.

\begin{definition}
	\label{definition:hg}
	Suppose $g \in \bz$ and $g \geq 2$.
	Define the $\hg$, {\em the stack of hyperelliptic genus $g$ curves}
	as the category fibered in groupoids over schemes, whose fiber over a scheme $B$
	corresponds tuples $(B,C,f,P,h)$
	where $C$ is a smooth proper curve of genus $g$ over $B$ with geometrically connected fibers,
	$h: P \to B$ is a smooth proper curve of genus $0$ over $B$ with geometrically connected fibers,
	and $f: C \to P$
	is a degree $2$ locally free morphism. 
	Morphisms $(B,C,f,P,h) \to (B',C',f',P',h')$ are morphisms $t: B \to B', r: C \to C', s: P \to P'$ making all squares in the diagram
	\begin{equation}
		\label{equation:}
		\begin{tikzcd} 
			C \ar {r}{r} \ar {d}{f} & C' \ar {d}{f'} \\
			P \ar {r}{s} \ar{d}{h} & P' \ar{d}{h'} \\
			B \ar{r}{t} & B'
	\end{tikzcd}\end{equation}
	fiber squares.
\end{definition}
Note that $\hg$ as defined above is indeed a stack because the $\isom$ presheaf between any two points is a sheaf and 
descent data for $(B,C,f,P,h)$ is effective; the effectivity of descent data holds because descent data for $C$ and $(P,h)$ are separately effective, as is descent data for the morphism $f$.

\begin{remark}
	\label{remark:}
	We note that hyperelliptic curves over fields as defined in
	\autoref{definition:hg}
	may fail to have a map to $\bp^1$, as they may have a map to a genus $0$ curve which is not isomorphic to $\bp^1$.
	However, over an algebraically closed field, any hyperelliptic curve has a map to $\bp^1$.
\end{remark}

\subsection{Showing $\hg$ is an algebraic stack}
\label{subsection:algebraic-stack}

To show $\hg$ is an algebraic stack, we will construct a smooth cover by a scheme.
This scheme will be an open subscheme of a certain linear system on a Hirzebruch surface,
which we now define.
\begin{notation}
	\label{notation:hirzebruch}
	For $n \in \bz_{\geq 0}$, let $\mathbb F_n := \bp\left( \sco_{\bp^1_\bz}(-n) \oplus \sco_{\bp^1_\bz} \right)$
denote the $n$th Hirzebruch surface over $\spec \bz$.
Let $E \subset \mathbb F_n$ denote the ``directrix'' section of $\mathbb F_n \to \mathbb P^1$ corresponding to the surjection 
$\sco_{\bp^1_\bz}(-n) \oplus \sco_{\bp^1_\bz} \to \sco_{\bp^1_\bz}(-n)$
and let $F$ denote a fiber of the map $\mathbb F_n \to \mathbb P^1$.
Letting $\ul{G}$ denote the constant group scheme associated to a group $G$, 
we claim there is an isomorphism $\ul{\bz \oplus \bz} \to \pic_{\mathbb F_n/\bz}$
sending $(1,0) \mapsto \sco(F)$ and $(0,1) \mapsto \sco(E)$. 
Indeed,
this is an isomorphism by the fibral isomorphism criterion 
\cite[17.9.5]{EGAIV.4}
because it induces an isomorphism over every geometric point of $\spec \bz$.
The intersection pairing on $\mathbb F_n$ satisfies $E \cdot E = -n, E \cdot F = 1, F \cdot F = 0$,
see \cite[Proposition IV.1]{beauville:complex-algebraic-surfaces}.
\end{notation}

To start, we show that every hyperelliptic curve of genus $g$ has a canonical immersion into a Hirzebruch surface, and lies in 
a particular linear system. 
This will allow us to check that the schematic cover of $\hg$ we construct maps surjectively to $\hg$.
\begin{lemma}
	\label{lemma:hg-cover}
	Any hyperelliptic curve $C$ of genus $g \geq 2$ over an algebraically closed field $k$ is a closed subscheme of
	$\mathbb F_{g+1}$ in the linear system $\sco(2E + (2g+2)F)$.
	Further, any curve in the linear system associated to 
	$\sco(2E + (2g+2)F)$ has genus $g$.
\end{lemma}
\begin{proof}
	Given a map $f: C \to \bp^1$ over $k$
	the surjective adjunction map $f^* f_* \sco_C \to \sco_C$ induces an map
	$C \to \bp \left( f_* \sco_C \right)$.
	This map is an immersion, as can be verified on fibers over $\bp^1$.
	We claim $f_* \sco_C \simeq \sco_{\bp^1}(-g-1) \oplus \sco_{\bp^1}$.
	We know $f_* \sco_C$ splits as a direct sum of line bundles because is a locally free sheaf on $\bp^1$.
	Because $1 = h^0(C, \sco_C) = h^0(\bp^1, f_* \sco_C)$, one of the summands of $f_* \sco_C$ must be $\sco_{\bp^1}$.
	To compute the other summand, we find the degree of $f_* \sco_C$.
	For sufficiently large $m$, we have
	\begin{align*}
	h^0(\bp^1, (f_* \sco_C)(m)) = h^0(\sco_C \otimes f^* \sco_{\bp^1}(m)) = h^0(f^* \sco_{\bp^1}(m)) = 2m-g+1.
	\end{align*}
	If $f_* \sco_C \simeq \sco_{\bp^1}(-n) \oplus \sco_{\bp^1}$ then we find 
	\begin{align*}
	h^0(\bp^1, (f_* \sco_C)(m)) = (m+1) + (m+1-n) = 2m+2-n
	\end{align*}
	and therefore $2-n = -g + 1$ and $n = g + 1$ as claimed.

	To conclude, it remains to check $C$ lies in the linear system $\sco(2E + (g+1)F)$.
	Write $C = \alpha E + \beta F$.
	Since $C \to \bp^1$ has degree $2$, the intersection $C \cdot E$ has degree $2$, which implies
	$\alpha= 2$.
	Since $C$ has genus $g$ and the canonical divisor of $\mathbb F_{g+1}$ is $-2E - (g+1+2)F$
	adjunction implies $2g-2 = ((-2E - (g+1+2)F) + (2E + \beta F)) \cdot (2E + \beta F)$.
	Solving for $\beta$ yields $\beta = 2g + 2$, as we wanted.
	The final statement that 
	any curve of class
	$\sco(2E + (2g+2)F)$ has genus $g$
	can be deduced directly from Riemann-Roch.
	Alternatively, one may deduce it from the fact that hyperelliptic curves of genus $g$ exist and that the genus is constant in flat connected families.
\end{proof}

The above lemma will allow us to show that a certain linear space of sections on a Hirzebruch surface is a cover
of the stack $\hg$.
Let $G := \aut_{\mathbb F_n/\bz}$ denote the automorphisms group scheme
of the Hirzebruch surface $\mathbb F_n$ over $\spec \bz$.
The cover of $\hg$ we construct will be a $G$ torsor, and so to show $\hg$ is an algebraic stack (so that it has a smooth cover
by a scheme) 
we will need to know $G$ 
is smooth.
\begin{lemma}
	\label{lemma:smooth-aut}
	For any $n \geq 1$ 
	$\aut_{\mathbb F_n/\bz}$ is isomorphic to a certain semi-direct product $\bg_a^{n+1} \rtimes (\gl_2/\mu_n)$.
	In particular, it
	is smooth and connected of relative dimension $n+5$.
\end{lemma}
\begin{proof}
	The plan is to construct a map from the smooth group scheme 
	$\bg_a^{n+1} \rtimes (\gl_2/\mu_n)$
	to $\aut_{\mathbb F_n/\bz}$, which we will verify to be an isomorphism.
	We now describe the semidirect product structure in terms of a group action on $H^0(\bp^1, \sco \oplus \sco(n))$.
	Choose $x,y$ as a basis for $H^0(\bp^1, \sco(1))$
	and represent elements of $H^0(\bp^1, \sco \oplus \sco(n))$ by pairs $(r, z(x,y))$ for $z$ a degree $n$ polynomial.
	Then $g \in \gl_2$ sends $(r,z(x,y)) \mapsto (r, z(gx,gy))$ and $(\alpha_0, \ldots, \alpha_n) \in \bg_a^{n+1}$ sends
	$(r,z(x,y)) \mapsto (r, z(x,y) + \sum_{i=0}^n \alpha_i x^iy^{n-i})$.
	The action of $\gl_2$ has kernel $\mu_n \subset \gl_2$.
	Now, because the pushforward of $\sco_{\mathbb F_{n}}(E + nF)$ along $\mathbb F_{n} \to \bp^1$ is 
	$\sco \oplus \sco(n)$
	the complete linear system associated to $nF + E$ induces a map
	$\mathbb F_n \to \bp H^0(\bp^1, \sco \oplus \sco(n))$.
	This contracts the directrix $E$ but is an immersion on the open complement of $E$.
	We claim the action of $\bg_a^{n+1} \rtimes (\gl_2/\mu_n)$ fixes the image $Y$ of $\mathbb F_{n}$ in $\bp H^0(\bp^1, \sco \oplus \sco(n))$.
	First we check the action of $\gl_2$ fixes the image.
	The action fixes both the point $p$ corresponding to the quotient $H^0(\bp^1, \sco \oplus \sco(n)) \to H^0(\bp^1, \sco)$ 
	and the rational normal curve $R$ lying in the hyperplane associated to the quotient $H^0(\bp^1, \sco \oplus \sco(n)) \to H^0(\bp^1, \sco(n))$.
	As $Y$ is the 
	cone over $R$ with cone point $p$, any such automorphism must preserve $Y$.
	The action of $\bg_a^{n+1}$ fixes $p$ and sends $R$ to another curve on the cone $Y$ not meeting $p$. It therefore preserves $Y$ as well.
	We obtain the induced map $\theta: \bg_a^{n+1} \rtimes (\gl_2/\mu_n) \to \aut_{\mathbb F_n/\bz}$ via the identification $\aut_{Y/\bz} \simeq \aut_{\mathbb F_n/\bz}$
	from blowing up $p$.

	It remains to check this map is an isomorphism.
By the fibral isomorphism criterion
\cite[17.9.5]{EGAIV.4},
since 
$\bg_a^{n+1} \rtimes (\gl_2/\mu_n)$ is flat over $\spec \bz$, in order to check $\theta: \bg_a^{n+1} \rtimes (\gl_2/\mu_n) \to \aut_{\mathbb F_{n}/\bz}$
is an isomorphism, it suffices to do so on geometric fibers.
%
%
We now check the restriction of $\theta$ to $\spec k$, for $k$ an algebraically closed field, is an isomorphism.
This may be checked using the above description of $G$ acting on $Y$,
the image of $\mathbb F_n \to \bp H^0(\bp^1, \sco \oplus \sco(n))$
induced by $\sco_{\mathbb F_{n}}(nF + E)$.
As mentioned above, we can identify $\aut_{Y/k} \simeq \aut_{\mathbb F_n/k}$.
Every automorphism of $Y$ must preserve the point $p$.
Therefore, every automorphism of $\mathbb F_n$ preserves $E$ and
we obtain a map $\aut_{Y/k} \simeq \aut_{\mathbb F_n/k} \to \aut_{E/k} = \pgl_2$.
Using that every automorphism of $Y$ is induced by one of 
$\bp H^0(\bp^1, \sco \oplus \sco(n))$ fixing $Y$,
this map is surjective, and the kernel can be identified with $\bg_a^{n+2} \times \bg_m$ by direct calculation. 
The latter factor corresponds to the central
torus in $\gl_2/\mu_n$.
This verifies that $\theta$ is an isomorphism.
\end{proof}

To state the next proposition proving that $\hg$ is a smooth algebraic stack, we now introduce a smooth scheme with a map to $\hg$.
Let $\pi: \mathbb F_{g+1} \to \spec \bz$ denote the projection and define
$U \subset \bp \left(\pi_*\sco_{\mathbb F_{g+1}}(2E+(2g+2)F) \right)$ over $\spec \bz$
as the open subscheme parameterizing 
smooth curves in the linear system
$\pi_*\sco_{\mathbb F_{g+1}}(2E+(2g+2)F$
	with universal family $C \to \bp^1_U \to U$.
	The family $C \to \bp^1_U \to U$ is equivariant for the action of $G := \aut_{\mathbb F_{g+1}/\bz}$
	and descends to a map of stacks $[C/G] \to [\bp^1_U/G] \to [U/G]$ inducing a map $[U/G] \to \hg$.
	
\begin{proposition}
	\label{proposition:smooth-hg-cover}
	For $g \geq 2$, the above constructed map $[U/G] \to \hg$ is an equivalence of stacks.
Further, $U \to \hg$ is a smooth cover and hence $\hg$ is a smooth integral algebraic stack
of relative dimension $2g - 1$ over $\spec \bz$.
\end{proposition}
\begin{proof}
	Note that the final statement follows from the first by \autoref{lemma:smooth-aut}
	because $U \to [U/G]$ is a smooth cover, and $[U/G]$ is an algebraic stack.

	Hence, it suffices to show $[U/G] \to \hg$ is an equivalence of stacks.
	In the statement of \autoref{proposition:smooth-hg-cover}
	we have constructed a map $[U/G] \to \hg$ and we now construct an inverse map.
	In order to construct a map $\hg \to [U/G]$, 
	given any $B \to \hg$, we wish to construct a $G$ torsor over $B$ with a map to $U$.
	The map $B \to \hg$ yields a family $C \xra{f} P \xra{h} B$ and we have an immersion $C \to \bp(f_*\sco_C)$ induced by the surjective adjunction map $f^* f_* \sco_C \to \sco_C$.
	Then, consider the scheme $I := \isom(\mathbb F_{g+1}, \bp(f_* \sco_C)) \to B$.
	To construct a map $I \to U$, it is enough to produce a family of smooth curves in the linear system $\sco(2E + (2g+2)F)$ in $(\mathbb F_{g+1})_I$.
	This is provided by the pullback of $C \to \bp(f_* \sco_C) \to B$ along $I \to B$, using the isomorphism
$(\mathbb F_{g+1})_I \simeq (\bp(f_* \sco_C))_I$ coming from the universal property of $I$.
	We next claim $I$ is a $G$-torsor over $B$.
	It certainly has a $G$ action via the action of $G$ on $\mathbb F_{g+1}$, and it is straightforward to verify
	the map $G \times_{\spec \bz} I \to I \times_B I, (g,x) \mapsto (x,gx)$ is an isomorphism.
	The cover $I \to B$ is therefore smooth by \autoref{lemma:smooth-aut}, and it is surjective by \autoref{lemma:hg-cover}.

	To conclude, it is enough to show that this map $\hg \to [U/G]$ defines an inverse equivalence to the map $[U/G] \to \hg$ from the statement.
	On objects, the composition $\hg \to [U/G] \to \hg$ constructs the immersion into a relative Hirzebruch surface, and then forgets it,
	so is equivalent to the identity.
	Further, the composition $[U/G] \to \hg \to [U/G]$ is also equivalent to the identity because the immersion into a 
	Hirzebruch surface is uniquely
	determined by the hyperelliptic curve, and any automorphism of the hyperelliptic curve induces a unique automorphism of the relative
	Hirzebruch surface.

	The final statement about the dimension holds because $G$ has relative dimension $g + 6$, using \autoref{lemma:smooth-aut}
	and $U$ has relative dimension $3g + 5$ by \autoref{lemma:linear-system-dimension} below.
	Therefore, $\dim \hg = (3g+5) - (g+6) = 2g - 1$.
\end{proof}

The following standard calculation on the dimension of a linear system was needed above to compute the dimension of $\hg$.
\begin{lemma}
	\label{lemma:linear-system-dimension}
	For $k$ a field and $g \geq 0$, $\bp H^0( \mathbb F_{g+1}, \sco_{\mathbb F_{g+1}}(2E + (2g+2)F))$ on $\mathbb F_{g+1}$ over $k$
	has dimension $3g + 5$.
\end{lemma}
\begin{proof}
	Let $C$ be a smooth curve in the linear system $2E + (2g + 2)F$.
	We obtain an ideal sheaf exact sequence
	\begin{equation}
		\label{equation:hirzebruch-ideal-sequence}
		\begin{tikzcd}
			0 \ar {r} & \sco_{\mathbb F_{g+1}} \ar {r} & \sco_{\mathbb F_{g+1}}(C) \ar {r} & \sco_C(C) \ar {r} & 0.
	\end{tikzcd}\end{equation}
	Note that $H^1(\mathbb F_{g+1}, \sco_{\mathbb F_{g+1}})= 0$ as can be computed by the Leray spectral sequence associated to 
	the composition $\mathbb F_{g+1} \xra{f} \bp^1 \to \spec k$ because $R^1 f_* \sco_{\mathbb F_{g+1}} = 0$ and $H^1(\bp^1, \sco_{\bp^1}) = 0$.
	Therefore, upon taking cohomology of
	\eqref{equation:hirzebruch-ideal-sequence}, we get an exact sequence
	\begin{equation}
		\label{equation:}
		\begin{tikzcd}
			0 \ar {r} & H^0(\mathbb F_{g+1}, \sco_{\mathbb F_{g+1}}) \ar {r} & H^0(\mathbb F_{g+1}, \sco_{\mathbb F_{g+1}}(C)) \ar {r} & H^0(\mathbb F_{g+1}, \sco_C(C)) \ar {r} & 0.
	\end{tikzcd}\end{equation}
	We are aiming to show $h^0(\mathbb F_{g+1}, \sco_{\mathbb F_{g+1}}(C)) = 3g + 6$, which has dimension $1$ more than the projectivization in the statement of the lemma.
	Since $h^0(\mathbb F_{g+1}, \sco_{\mathbb F_{g+1}}) = 1$, it is enough to show 
	$h^0(\mathbb F_{g+1}, \sco_C(C)) = 3g +5$.
	We have
	\begin{align*}
		\deg \sco_C(C) =  (2E + (2g+2)F) \cdot (2E + (2g+2)F) = -4(g+1) + 4(2g+2) = 4g + 4.
	\end{align*}
	Since $C$ has genus $g$ by \autoref{lemma:hg-cover},
	it follows from Riemann Roch that
	$h^0(\mathbb F_{g+1}, \sco_C(C)) = 4g + 4 - g + 1 = 3g + 5$.
\end{proof}

\subsection{Showing $\hg$ has a closed immersion into $\mg$}
\label{subsection:closed-immersion}
We next check the natural map $\hg \ra \mg$ is a closed immersion.
The following general lemma will be useful.
\begin{lemma}
	\label{lemma:closed-immersion-criterion}
	Suppose $f: X \to Y$ is a proper morphism of algebraic stacks with diagonals that are separated and of finite type.
	Assume $f$ induces a bijection on isotropy subgroups at every geometric point of $X$, is injective on geometric points, 
	and injective on tangent vectors at every geometric point.
	Then $f$ is a closed immersion.
\end{lemma}
With some additional work, one can remove the hypotheses on the diagonals of $X$ and $Y$, though we will not need this generalization.

\begin{proof}
	Let $Y' \to Y$ be a smooth cover by a scheme, let $X' := X \times_Y Y'$, and let $f' : X' \to Y'$ denote the base change of $f$.
	It is enough to check $f'$ is a closed immersion.
	Because $f$ is an injection on isotropy groups, 
	and $X$ and $Y$ have separated finite type diagonals,	
	$f$ is representable by
	\cite[Theorem 2.2.5(1)]{Conrad:arithmeticModuliOf}.
	Thus, $X'$ is an algebraic space.
	Note that $X'$ is in fact a scheme because $f'$ is quasi-affine by a version of Zariski's main theorem
	\cite[\href{https://stacks.math.columbia.edu/tag/082J}{Tag 082J}]{stacks-project}.
	Further, because $f$ is a surjection on isotropy groups and $f$ is injective on geometric points, 
	$f'$ is also injective on geometric points.
	We find that $f'$ is proper, injective on geometric points, and injective on tangent vectors, 
	hence a closed immersion \cite[18.12.6]{EGAIV.4}.
\end{proof}

	\begin{lemma}
		\label{lemma:closed-immersion-hg}
		For $g \geq 2$, the natural map $\hg \to \mg$ sending $(B,C,f,P,h)$ over a scheme $B$ to the curve $C$ over $B$ is a closed immersion.
	\end{lemma}
	\begin{proof}
	The idea is to verify the conditions of \autoref{lemma:closed-immersion-criterion}.
	To check that $\hg \ra \mg$ is a closed immersion, 
	we first verify that the map induces bijections on isotropy group spaces at each geometric point. 
	In particular, this will imply the map is representable.
	As a preliminary step, we claim given a hyperelliptic curve $C$ of genus $g \geq 2$ over an algebraically closed field $k$ and an automorphism $r: C \to C$, there is a unique map $f: C \to \bp^1$,
	up to automorphism of $\bp^1$. 	
	To see this, it is enough to verify that for any degree $2$ map $C \to \bp^1$, the composition with the $(g-1)$-Veronese $\bp^1 \to \bp^{g-1}$ realizes the map $C \to \bp(H^0(C, \omega_C)) \simeq \bp^{g-1}$
	associated to the invertible sheaf $\omega_C$. This holds because $\omega_C$ is the unique invertible sheaf on $C$ of degree $2g-2$ with a $g$-dimensional space of global sections.
	Fixing such a map $C \to \bp^1$, we find that $\bp^1$ is the quotient of $C$ by the hyperelliptic involution as a scheme.
	Therefore, by the universal property of quotients,
	there is a unique automorphism $s: \bp^1 \to \bp^1$ over $k$ so that $f \circ r = s \circ f$.
	This shows the map on isotropy groups is surjective, and has kernel supported on a single geometric point. To show the kernel is trivial, it is
	enough to show the kernel is reduced over $k$. The same argument as above applied over the dual numbers in place of $k$ shows the kernel is indeed reduced.
	
	Next, we check $\hg \to \mg$ is injective on geometric points and tangent vectors at each geometric point.
	The fiber over an map $\spec k \to \mg$ for $k$ an algebraically closed field,
	corresponding to a curve $C$ over $\spec k$
	is identified with the scheme $W$ 
	parameterizing degree $2$ line bundles on $C$ which have a $2$-dimensional space of global sections, up to isomorphism.
	First, $W$ is either empty or set theoretically a single point, because
	given any curve $C$ of genus at least $2$ over an algebraically closed field, there is at most one degree $2$ map to $\bp^1$,
	up to automorphisms of $\bp^1$.
	It remains to check $W$ is reduced in the case it is nonempty.
	The tangent space to the unique point of $W$ corresponding to the line bundle $\scl$ on $C$
	can be identified with the cokernel of the multiplication map
	$\mu_0 : H^0(C, \scl) \otimes H^0(C, \omega_C \otimes \scl^\vee) \to H^0(C, \omega)$
	\cite[Proposition 4.2(i)]{ACGH:I}, so it is enough to verify $\mu_0$ is surjective.
	We verify this standard calculation below in \autoref{lemma:dual-tangent-space-description}.

	To conclude, by \autoref{lemma:closed-immersion-criterion},
	it suffices to check $\hg \ra \mg$ is proper.
	For this, we use the valuative criterion for properness.
	Let $R$ be a discrete valuation ring, let $s$ denote the closed point of $\spec R$ and let $\eta$ denote the generic point. 
	We begin with a curve $h: C \to \spec R$ so that the generic fiber has a degree $2$ map to a genus $0$ curve $P$.
	It is enough to show there is some extension of $R$ on which $C$ factors through a degree $2$ map to $\bp^1_R$.
	By replacing $R$ with a suitable extension, we may assume we have a factorization $C_\eta \to \bp^1_\eta \to \eta$.
	Therefore, it is enough to show that given $C \to \spec R$ with $C_\eta \to \bp^1_\eta$ there is a unique extension of this map to a map $C \to \bp^1_R \to \spec R$.
	The pullback of $\sco_{\bp^1_\eta}(1)$ to $C_\eta$ gives a degree $2$ invertible sheaf on $C_\eta$
	which extends uniquely to a degree $2$ invertible sheaf $\scl$ on $H$, for example by properness of $\pic_{C/R}$.
	Riemann Roch (using $g \geq 2)$ and upper semicontinuity of cohomology together imply  $h^0(C_s, \scl_s) = 2$.
	Therefore, Grauert's theorem tells us $h_* \scl$ is locally free, hence free, of rank $2$.
	Then, up to elements of $R^\times$, there is then a unique choice of basis for $h_* \scl$ inducing a map
	$C \to \bp^1_R$ compatible with the given map $C_\eta \to \bp^1_\eta$, as we wished to show.
\end{proof}

\section{The general setup for checking injectivity on tangent vectors}
\label{section:general-setup}

Let $k$ be a field and let $C$ be a hyperelliptic curve over $k$. 
To understand whether the map $\phi_g: \hg \ra \ag$ is injective on tangent vectors, we want
to understand the composition
\begin{align*}
	T_{[C]} \hgb k \xra{T_{[C]} \iota_{g,k}} T_{[C]} \mgb k \xra{T_{[C]}\tau_{g,k}} T_{[\tau_g(C)]} \agb k.
\end{align*}
We have already explicitly described $T_{[C]}\tau_{g,k}$ by identifying it as dual to
\eqref{equation:torelli-dual-map} 
(see also \autoref{remark:hyperelliptic-image-char-0} and \autoref{remark:hyperelliptic-image-char-2} below for explicit descriptions of \eqref{equation:torelli-dual-map} in terms of differentials)
so we next want to understand the image of $T_{[C]}\iota_{g,k}$.
Following 
\cite[Chapter 21, \S5-\S6]{arbarelloCG:geomtry-of-algebraic-curves-ii}
we can identify $T_{[C]}\iota_{g,k}$ as follows.

Let $C$ be a hyperelliptic curve as above and $\scl$ the unique isomorphism class of invertible sheaf giving
rise to a hyperelliptic map $C \ra \bp^1$.
This $\scl$ has nontrivial automorphisms, but the automorphisms will be irrelevant for the ensuing computations.
Let
\begin{align*}
	\mu_0 : H^0(C, \scl) \otimes H^0(C, \omega_C \otimes \scl^\vee) \ra H^0(C, \omega_C)
\end{align*}
denote the multiplication map.
Then, 
as in \cite[Chapter 21, (6.1)]{arbarelloCG:geomtry-of-algebraic-curves-ii},
there is a canonical map 
\begin{align}
	\label{equation:mu-1}
	\mu_1 : \ker \mu_0 \ra H^0(C, \omega_C^{\otimes 2}).
\end{align}
To describe this map explicitly, recall that the differential is a map
$\sco_C \to \Omega_C = \omega_C$.
This induces a map $\scl \to \omega_C \otimes \scl$,
and hence a map $H^0(C, \scl) \to H^0(C, \omega_C \otimes \scl)$
which sends $r \mapsto dr$.
One can describe \eqref{equation:mu-1} explicitly as the map sending $r \otimes s \mapsto dr \cdot s$.
It is not obvious this is well-defined, but the well definedness along with this description is verified in \cite[Chapter 21, p. 810]{arbarelloCG:geomtry-of-algebraic-curves-ii}.
(The language used there makes it seem like they are working over $\bc$, but their proof works equally well over any field.)

\begin{proposition}
	\label{lemma:dual-tangent-space-description}
	For $C$ a hyperelliptic curve,
	the composition 
	\begin{align}
		\label{equation:tangent-space-composition}
		T_{[C]} \hgb k \ra T_{[C]} \mgb k \ra T_{[\tau_{g}(C)]} \agb k
	\end{align}
is dual to the composition
\begin{align}
	\label{equation:ag-to-hg-dual}
	H^0(C, \omega_C^{\otimes 2})/\im \mu_1 \xleftarrow{(T_{[C]}\iota_{g,k})^\vee} H^0(C, \omega_C^{\otimes 2}) \xleftarrow{(T_{[C]}\tau_{g,k})^\vee} \sym^2 H^0(C, \omega_C).
\end{align}
Hence, the dimension of the kernel of \eqref{equation:tangent-space-composition} agrees with the dimension of the cokernel of \eqref{equation:ag-to-hg-dual}.
In particular, \eqref{equation:tangent-space-composition} is injective if and only if \eqref{equation:ag-to-hg-dual} is surjective.
Further, $\coker \mu_0 = 0, \dim \im \mu_1 = g - 2,$ and $\dim \im (T_{[C]} \tau_{g,k})^\vee = 2g-1$.
\end{proposition}
\begin{proof}
Let $C$ a smooth projective geometrically connected curve and $L$ an invertible sheaf on $C$.
Then, $C$ has a rank $2$ locally free sheaf, $\Sigma_L$, as defined in \cite[p. 804]{arbarelloCG:geomtry-of-algebraic-curves-ii},
such that $H^1(C, \Sigma_L)$ parameterizes first order deformations of the pair $(C,L)$
\cite[Chapter 21, Proposition 5.15]{arbarelloCG:geomtry-of-algebraic-curves-ii}.
Further, as described in \cite[Chapter 21, (5.24)]{arbarelloCG:geomtry-of-algebraic-curves-ii},
there is a natural map $\mu : H^0(C, L) \otimes H^0(C, \omega_C \otimes L^\vee) \ra H^0(C, \omega_C \otimes \Sigma_L^\vee)$.
The key property of $\mu$ is that if $\alpha \in H^1(C,\Sigma_L)$ corresponds to a first order 
deformation $(\scc, \scl)$ of $(C,L)$,
then all sections of $L$ extend to sections of $\scl$ (meaning that, if $\pi: \scc \ra \spec k[\varepsilon]/(\varepsilon^2)$ is the structure map,
$\pi_* \scl$ is locally free,)
if and only if $\alpha \in (\im \mu)^\perp$ \cite[Chapter 21, Proposition 5.26]{arbarelloCG:geomtry-of-algebraic-curves-ii}.
Here $(\im \mu)^\perp$ denotes the orthogonal subspace to $\im \mu$ under the
Serre duality pairing $H^0(C, \omega_C \otimes \Sigma_L^\vee) \otimes H^1(C, \Sigma_L) \ra k$.

Suppose $C$ is a hyperelliptic curve with corresponding line bundle $L$ defining the hyperelliptic map $C \ra \bp^1$.
Given $\beta \in H^1(C, T_C)$, we obtain a deformation of $C$, corresponding to a curve $\pi: \scc \ra \spec k[\varepsilon]/(\varepsilon^2)$.
Our main goal is to show that $\beta$ corresponds to a hyperelliptic curve precisely when $\beta \in (\im \mu_1)^\perp$.
By definition, if $[\scc]$ is hyperelliptic, there is an invertible sheaf $\scl$ on $\scc$ with $\pi_* \scl$ locally free of rank $2$.
As described above, using
\cite[Chapter 21, Proposition 5.15]{arbarelloCG:geomtry-of-algebraic-curves-ii},
we obtain that the pair $(\scc, \scl)$ corresponds to some $\alpha \in H^1(C, \Sigma_L)$
with $\alpha \in (\im \mu)^\perp$.

We next show that if $\beta$ corresponds to a hyperelliptic curve, then $\beta \in (\im \mu_1)^\perp$. 
In other words, we will show $T_{[C]} \hgb k \subset \left( \im \mu_1 \right)^\perp$.
By \cite[Chapter 21, Proposition 5.15]{arbarelloCG:geomtry-of-algebraic-curves-ii}, there is a natural map of sheaves
$\Sigma_L \ra T_L$ inducing the map $\sigma : H^1(C, \Sigma_L) \ra H^1(C, T_L)$ which sends the deformation $[(\scc, \scl)]$ to the deformation $[\scc]$.
In particular, $\sigma(\alpha) = \beta$.
Let $(T_{[C]}\tau_{g,k})^\vee: H^0(C, \omega_C^{\otimes 2}) \ra H^0(C, \omega_C \otimes \Sigma_L^\vee)$ denote the map which is Serre dual to $\sigma$.
By definition of $\mu_1$ (see \cite[Chapter 21, (6.1)]{arbarelloCG:geomtry-of-algebraic-curves-ii})
we have a commutative diagram
\begin{equation}
	\label{equation:mu-1-square}
	\begin{tikzcd} 
		H^0(C, L)\otimes H^0(C, \omega_C \otimes L^\vee) \ar {r}{\mu}  & H^0(C, \omega_C \otimes \Sigma_L^\vee)  \\
		\ker \mu_0 \ar {r}{\mu_1}\ar{u} & H^0(C, \omega_C^{\otimes 2}) \ar{u}{\sigma^\vee}.
\end{tikzcd}\end{equation}
For any $a \in \im \mu$, we have $\langle a, \alpha \rangle = 0$, where $\langle , \rangle $ denotes the pairing from Serre duality.
Therefore, for $b \in \im \mu_1$, commutativity of \eqref{equation:mu-1-square} implies $\langle \sigma^\vee(b), \alpha \rangle = 0$.
Functoriality of Serre duality then implies $0= \langle \sigma^\vee(b), \alpha \rangle  = \langle b, \sigma(\alpha) \rangle = \langle b,\beta \rangle$, and hence
$\beta \in \left( \im \mu_1 \right)^\perp$.

We claim that $\dim \ker \mu_0 = g-2, \dim \im (T_{[C]} \tau_{g,k})^\vee = 2g-1,$ and $\mu_1$ is injective. 
We now explain why verifying these three claims finishes the proof.
First, if $\dim \ker \mu_0 = g - 2$ then $\coker \mu_0 = 0$ because $\mu_0$ is a map from a $2g - 2$ dimensional vector space to a $g$ dimensional vector space.
We next explain why these claims imply $\dim \im \mu_1 = g-2$.
We know $T_{[C]} \hgb k$ is $2g-1$ dimensional (as $\hgb k$ is smooth of dimension $2g-1$ 
	\autoref{proposition:smooth-hg-cover}) and 
$T_{[C]}\hgb k \subset (\im \mu_1)^\perp$ inside the $3g-3$ dimensional vector space $H^0(C, T_C)$, 
as shown above.
Since $\dim \im \mu_1 = g-2$, by dimensional considerations, the containment $T_{[C]} \hgb k \subset \left( \im \mu_1 \right)^\perp$ must be an equality.
Therefore, the inclusion $T_{[C]} \hgb k \hookrightarrow H^0(C, T_C)$ is dual to the surjection $H^0(C, \omega_C^{\otimes 2}) \twoheadrightarrow H^0(C, \omega_C^{\otimes 2})/\im \mu_1$.

It remains to check $\dim \ker \mu_0 = g-2, \dim \im (T_{[C]} \tau_{g,k})^\vee = 2g-1,$ and $\mu_1$ is injective. 
First, let us check $\dim \ker \mu_0 = g-2$.
Because $C$ is hyperelliptic, we have $\omega_C = L^{\otimes (g-1)}$. In particular, $\omega_C \otimes L^\vee \simeq L^{\otimes (g-2)}$
and so $\mu_0$ is given by the map $H^0(C, L) \otimes H^0(C, L^{\otimes (g-2)}) \ra H^0(C, L^{\otimes (g-1)})$.
Letting $\left\{ 1,x \right\}$ denote a basis for $H^0(C, L)$, we see $\left\{ 1, x, \ldots, x^i \right\}$ is a basis
for $H^0(C, L^{\otimes i})$.
Applying this when $i = g-2$ and $i = g-1$, we find $\mu_0$ is surjective, so $\dim \ker \mu_0 = g-2$.
Explicitly, we see $\ker \mu_0$ is generated by $\left\{ 1 \otimes x^j - x \otimes x^{j-1} \right\}_{1 \leq j \leq g-2}$.

Next, we verify $\dim \im (T_{[C]} \tau_{g,k})^\vee = 2g-1$. As in the previous paragraph, we may choose a basis $\left\{ 1, x \right\}$ for $H^0(C, L)$ so that
$1, x, \ldots, x^{g-1}$ is a basis for $H^0(C, \omega_C)$. Then, the image of the multiplication map $(T_{[C]} \tau_{g,k})^\vee$ is spanned by
$1, x, \ldots, x^{2g-2}$, which has dimension $2g-1$.

To conclude, we show $\mu_1$ is injective.
Recall $\left\{ 1 \otimes x^j - x \otimes x^{j-1} \right\}_{1 \leq j \leq g-2}$ spans $\ker \mu_0$, as shown above,
and recall that $\mu_1$ is given by $r \otimes s \mapsto dr \otimes s$
\cite[Chapter 21, (6.6)]{arbarelloCG:geomtry-of-algebraic-curves-ii}, (alternatively, see \cite[p. 813-p. 814]{arbarelloCG:geomtry-of-algebraic-curves-ii}).
Therefore, $\im \mu_1$ is spanned by $dx \cdot x^{j-1}$ for $1 \leq j \leq g-2$, which are independent elements of $H^0(C, \omega_C \otimes L^{\otimes (g-1)}) \simeq H^0(C, \omega_C^{\otimes 2})$.
Hence, $\mu_1$ is injective.
\end{proof}

\section{Hyperelliptic curves in characteristic not 2}
\label{section:char-not-2}
The key to analyzing the map induced by $\phi_{g,k}$ on tangent spaces for $\chr(k) \neq 2$ is \autoref{proposition:char-0-tangent-space} below.
Let $k$ be an algebraically closed field of characteristic $p$ (allowing $p = 0$)
with $p \neq 2$.
Before proving injectivity, we set up some notation.

\subsection{Hyperelliptic differentials in characteristic not $2$}
	\label{remark:hyperelliptic-image-char-0}
Every hyperelliptic curve $C$ over an algebraically closed field $k$ of characteristic not $2$
can be expressed as the proper regular model of the affine curve $y^2 = f$
for $f \in k[x]$ a polynomial of degree $2g + 1$ with no repeated roots.
We can choose a basis of differentials for $C$ of the form
\begin{equation}
	\frac{dx}{y} , \frac{x \, dx}{y}, \ldots, \frac{x^{g-1} \, dx}{y}.
	\label{equation:char0-diff-basis}
\end{equation}
where here $x$ and $y$ are viewed as rational functions and $dx$ is viewed as a rational section of $H^0(C, \omega_C)$.

In the above basis, the multiplication map \eqref{equation:torelli-dual-map} above (which is dual to $T_{[C]} \mgb k \ra T_{[\tau_g(C)]} \agb k$)
has image
	\begin{align}
		\label{equation:image-char-0-mult-basis}
		\frac{(dx)^2}{y^2}, \ldots, \frac{x^{2g-2}\,(dx)^2}{y^2}.
	\end{align}
	Written another way, the basis is
	$\frac{1}{f}\,(dx)^2, \ldots, \frac{x^{2g-2}}{f}\,(dx)^2$.
	This is just what one obtains by multiplying together pairs of functions from the
	above described basis \eqref{equation:char0-diff-basis}.
	In particular, the image is a $2g-1$ dimensional subspace
	of the $3g-3$ dimensional vector space
	$H^0(C, \omega_C^{\otimes 2})$.

\subsection{Computing the tangent map in characteristic not $2$}
We now use our explicit description of the differentials to show 
\eqref{equation:ag-to-hg-dual} is surjective when the characteristic is not $2$.
\begin{lemma}
	\label{proposition:char-0-tangent-space}
	For $C$ a hyperelliptic curve over an algebraically closed field $k$ of characteristic $p \neq 2$, the composition \eqref{equation:ag-to-hg-dual}
	is surjective.
\end{lemma}
\begin{proof}
	We have an explicit understanding of the image of \eqref{equation:torelli-dual-map}
from \autoref{remark:hyperelliptic-image-char-0}.
If we can also explicitly describe $\im \mu_1$, then we will be able to determine
surjectivity of the composite map \eqref{equation:ag-to-hg-dual}.
To start, let's describe $\mu_0$, using the notation from \autoref{remark:hyperelliptic-image-char-0}.
Letting $\{1, x\}$ denote a basis of $\scl$, we find that $\{\frac{dx}{y},\frac{x\,dx}{y} \ldots, \frac{x^{g-2} \,dx}{y}\}$ is a basis for $H^0(C, \omega_C \otimes \scl^\vee)$.
Then, it follows that
\begin{align}
	\label{equation:ker-mu-0-char-0}
	x \otimes \frac{dx}{y} - 1 \otimes \frac{x \, dx}{y}, x \otimes \frac{x \, dx}{y} - 1 \otimes \frac{x^2 \, dx}{y}, \ldots, 
	x \otimes \frac{x^{g-3}\, dx}{y} - 1 \otimes \frac{x^{g-2} \, dx}{y}
\end{align}
is a basis for $\ker \mu_0$.

The map $\mu_1$ is given explicitly
by sending $\alpha \otimes \beta \mapsto d\alpha \cdot \beta$
\cite[Chapter 21, (6.6)]{arbarelloCG:geomtry-of-algebraic-curves-ii}.
Therefore, applying $\mu_1$ to \eqref{equation:ker-mu-0-char-0}, we find $\im \mu_1$ is generated by 
\begin{align}
	\label{equation:im-mu1-char-0}
	\frac{(dx)^2}{y},\frac{x(dx)^2}{y} \ldots, \frac{x^{g-3} (dx)^2}{y}.
\end{align}
On the other hand, the image of the map $H^0(C, \omega_C^{\otimes 2}) \leftarrow \sym^2 H^0(C, \omega_C)$ from \eqref{equation:torelli-dual-map}
is given in \eqref{equation:char0-diff-basis}.
Since together, the union of the $2g-1$ elements of \eqref{equation:char0-diff-basis}
and the $g-2$ elements of \eqref{equation:im-mu1-char-0} span $H^0(C, \omega_C^{\otimes 2})$,
it follows that the composite map
\eqref{equation:ag-to-hg-dual} is surjective.
\end{proof}

\section{Hyperelliptic curves in characteristic $2$}
\label{section:char-2}

As in \autoref{section:char-not-2}, to check injectivity of $\phi_g$ on tangent vectors, we may assume our base field $k$
is algebraically closed.
To conclude the proof of \autoref{theorem:hyperelliptic}, we only need prove \autoref{proposition:char-2-tangent-space} below.
We now set up notation for the proof.
The key difference in characteristic $2$ is that hyperelliptic curves
cannot be described in terms of an equation of the form $y^2 = f$, for $f \in k[x]$, of degree more than $1$, as any such curve
would be singular at the roots of $\frac{\partial f}{\partial x}$.
We now describe a general form for hyperelliptic curves in characteristic $2$.

\subsection{Equations for hyperelliptic curves in characteristic 2}
	\label{remark:char-2-form}
We start by reviewing a standard normal form for hyperelliptic curves in characteristic $2$.
This is stated in 
\cite[Notation 1.1]{elkinP:ekedahl-oort-strata-of-hyperelliptic-curves-in-characteristic-2},
and we provide some more details here.

\begin{lemma}
	\label{lemma:char-2-form}
	Over an algebraically closed field $k$ of characteristic $2$, every hyperelliptic curve
of genus at least $2$ can be written as the projective regular model of a curve of the form 
$y^2 - y = f$, for $f \in k(x)$. 
For a general such curve, $f$ can be chosen in the form
\begin{align}
	\label{equation:char-2-form}
f = \alpha_0 x + \frac{\alpha_1}{x-a_1} + \cdots + \frac{\alpha_g}{x-a_g} 
\end{align}
with $\alpha_0, \ldots, \alpha_g, a_1, \ldots, a_g \in k$.
The corresponding curve $C$ is ramified over $\bp^1$ at the preimages of $a_1, \ldots, a_g$, and $\infty$,
with ramification order $2$ at each such point.
\end{lemma}
\begin{remark}
	\label{remark:}
	A hyperelliptic curve can be written in this form precisely if it is ordinary, though we will not need this fact.
\end{remark}
\begin{proof}
	To start, we claim the curve can be written in the form $ay^2 + by + c = 0$ with $a \in k, b \in k[x], c \in k[x]$.
	To see this, given a hyperelliptic curve $C \to \bp^1$, we know from \autoref{lemma:hg-cover}
	it can be described as a closed subscheme of $\mathbb F_{g+1}$ in the linear system $2E + (2g+2)F$.
	This implies the curve can be written as $\beta_0 y^2 + \beta_1 yz + \beta_2 z^2$ with $\beta_i \in H^0(\bp^1, \sco(i \cdot (g+1)))$.
	Restricting this to $\ba^1 \subset \bp^1$ gives the claim.	

	Next, we show one may change variables to put a generic such curve in the form of \eqref{equation:char-2-form}.
	By our generality assumption, we may assume $b$ has $g + 1$ distinct roots.
	Then, by applying a change of variables, we may send one of the roots of $b$ to $\infty$, and hence assume that $b$ has degree $g$ as a polynomial in $x$ and has $g$ distinct roots.
	By scaling $y$ by $1/\sqrt{a}$, we may assume $a = 1$ and the curve is given by $y^2 + byz + cz^2$, or equivalently just $y^2 + by + c$ in affine coordinates.
	Next, we may replace $y$ by $by$ and divide the whole equation by $b^2$ so as to assume that
	the equation is of the form $y^2 - y = \frac{c}{b^2}$.
	We next
	apply a partial 
	fraction decomposition for $\frac{c}{b^2}$
so as to write $\frac{c}{b^2} = c_2 x^2 + c_1x + c_0 + \sum_i \frac{p_i}{(x-a_i)^{2r_i}}$, where $b = \prod (x-a_i)^{r_i}$ for distinct $a_i$ and $p_i$ are polynomials in $x$ of degree at most $2r_i$.
	Observe that the change of variables $y \mapsto \frac{y+1}{(x-a_i)^{r_i}}$
	sends $y^2 - y \mapsto \left( \frac{y}{(x-a_i)^{r_i}} \right)^2 - \frac{y}{(x-a_i)^{r_i}} + \left( \frac{1}{\left( x-a_i \right)^{2r_i}} - \frac{1}{(x-a_i)^{r_i}} \right)$.
	Hence, renaming $\frac{y}{(x-a_i)^{r_i}}$ to $y$, we can perform such changes of variables to 
	cancel out the highest even power of any denominator appearing in the partial fraction decomposition of $\frac{c}{b^2}$.
	By a similar change of variables, we may modify the quadraic and linear terms to assume $c_2 = c_0 = 0$.
	Therefore, in the case that $a,b,c$ are general so that $r_i = 1$ for all $i$, it follows that $f$ may be written in the form \eqref{equation:char-2-form}.
	
	To conclude, it remains to compute the ramification points and their ramification orders.
	A general hyperelliptic curve is given by 
	$\beta_0 y^2 + \beta_1 yz + \beta_2 z^2$ with $\beta_i \in H^0(\bp^1, \sco(i \cdot (g+1)))$
	as above, so that $\beta_1$ has $g+1$ distinct roots.
	Smoothness of the curve implies $\beta_1$ and $\beta_2$ have no common roots.
	We can also assume that one of the roots is over $\infty \in \bp^1$.
	Then, for $b \in k[x]$ the dehomogenized version of $\beta_1$, $b$ has $g$ distinct simple roots. By computing the derivative of
	$\beta_0 y^2 + \beta_1 yz + \beta_2 z^2$ with respect to $y$, we see $C \to \bp^1$ is precisely ramified at the roots of $\beta_1$.

	We claim that under the setup of the previous paragraph,
	the ramification orders at the preimages of $a_1, \ldots, a_g, \infty$ is $2$, and $C \to \bp^1$ has no other ramification points.
	To see this, because we have a degree $2$ map and we are in characteristic $2$, the ramification order is at least $2$ at each such point.
	Observe also that the sum of the ramification orders at all ramified points is 
	$\deg \Omega_{C/\bp^1} = 2g + 2$
	and there are $g+1$ points $a_1, \ldots, a_g$, and $\infty$.
	Hence, each point must have ramification order exactly $2$, and these must account for all the ramification points of $C \to \bp^1$.
\end{proof}

\subsection{Hyperelliptic differentials in characteristic $2$}
	\label{remark:hyperelliptic-image-char-2}
Now let's proceed to explicitly write down the differentials, as in the characteristic $0$ case.

	Let $k$ be a field of characteristic $2$ and 
	let $C$ be a hyperelliptic curve given as the projective regular model of the curve defined by the equation $y^2 - y = f$, for 
	$f = \alpha_0 x + \frac{\alpha_1}{x-a_1} + \cdots + \frac{\alpha_g}{x-a_g} \in k(x).$
	As described in \autoref{lemma:char-2-form}, $\pi: C \ra \bp^1$ is ramified at the preimages of 
	the points $a_1, \ldots, a_g$ and $\infty$
	(where $\infty$ is the point in the complement of $\spec k[x] \simeq \ba^1 \subset \bp^1$).

	By the assumption
	that $f$ is as in \eqref{equation:char-2-form}, it follows from
	\autoref{lemma:char-2-form} that
	there are $g+1$ ramification points of $\pi$, and the differential $dx$ is ramified to order $2$ at 
	the preimages of $V(x-a_1), \ldots, V(x-a_g)$.
	In other words, the relative sheaf of differentials $\Omega_\pi$ is a skyscraper sheaf at $\pi^{-1}(V(x - a_i))$ with degree $2$ at $\pi^{-1}(V(x - a_i))$,
	for $1 \leq i \leq g$.
	Since $\deg dx = 2g - 2$, we conclude that $dx$ has a pole of order $2$ at the unique point over $\infty \in \bp^1$.
	Further, the function $x-a_i$ vanishes to order $2$ at the preimage of $\pi^{-1}(V(x-a_i))$ and therefore has a pole of order $2$ at $\pi^{-1}(\infty)$.
	It follows that the functions 
	\begin{align*}
	\omega_1 := \frac{dx}{x-a_1}, \ldots, \omega_g := \frac{dx}{x-a_g}
	\end{align*}
	form a basis of $H^0(C, \omega_C)$.

	In the above basis, the map \eqref{equation:torelli-dual-map}
	\begin{align*}
		H^0(C, \omega_C^{\otimes 2}) \leftarrow \sym^2 H^0(C, \omega_C).
	\end{align*}
has image spanned by elements of the form 
	\begin{align}
		\label{equation:image-char-2-mult-basis}
		\frac{(dx)^2}{(x-a_i)(x-a_j)}
	\end{align}
	Of course, these elements will not be independent, but they will necessarily span a $2g-1$ dimensional subspace of 
	$H^0(C, \omega_C^{\otimes 2})$ by \autoref{lemma:dual-tangent-space-description}.

\subsection{Computing the map on tangent spaces in characteristic $2$}

\begin{lemma}
	\label{proposition:char-2-tangent-space}
	For $C$ a hyperelliptic curve of genus $g \geq 2$ over an algebraically closed field $k$ of characteristic $2$, 
	the composition \eqref{equation:ag-to-hg-dual}
	has cokernel of dimension $g -2$.
\end{lemma}

\begin{proof}
	For $C$ any hyperelliptic curve, with notation as in \autoref{lemma:dual-tangent-space-description}, $\dim \im \mu_1 = g-2$ and $\dim \im (T_{[C]}\tau_g)^\vee = 2g - 1.$
	It follows that the composition \eqref{equation:ag-to-hg-dual} has image of dimension at least $(2g-1) - (g-2) = g + 1$.
	Equivalently, 
	\eqref{equation:ag-to-hg-dual}
	has cokernel of dimension at most $g-2$, since $\dim \hgb {\mathbb F_2} = 2g-1 = (g+1)+(g-2)$.
	Therefore, by upper semicontinuity of cokernel dimension, in order to show 
	\eqref{equation:ag-to-hg-dual}
	has cokernel of dimension $g-2$ at all $[C] \in \hgb {\mathbb F_2}$ over a
	field of characteristic $2$, it suffices to show this holds over a general such $[C] \in \hgb {\mathbb F_2}$.

	As described in \autoref{remark:char-2-form}, a general hyperelliptic curve over $k = \ol k$ with $\chr(k) = 2$ can be written in the form 
$y^2 - y = f$, for 
$f = \alpha_0 x+ \frac{\alpha_1}{x-a_1} + \cdots + \frac{\alpha_g}{x-a_g} \in k(x)$,
with elements
$\alpha_0, \ldots, \alpha_g, a_1, \ldots, a_g \in k$ and with $a_1, \ldots, a_g$ distinct.
Following the strategy outlined in \autoref{section:general-setup},
we next determine $\ker \mu_0$ and then $\im \mu_1$ for such curves $C$.

We can choose $1, x-a_1$ as a basis for $H^0(C, \scl)$ identifying $\scl \simeq \sco_C(2 \cdot \pi^{-1}(\infty))$, for $\pi: C \ra \bp^1$ the hyperelliptic map. 
For $1 < i \leq g$, the differential forms $\frac{1}{a_1 - a_i} \left( \omega_1 - \omega_i \right) = \frac{dx}{(x-a_1)(x-a_i)}$ have a zero of order $2$ at $\infty$,
and therefore lie in $H^0(C, \omega_C \otimes \scl^\vee)$.
Because they are also independent,
it follows that
\begin{align}
	\label{equation:char-2-less-l}
	\frac{dx}{(x-a_1)(x-a_2)}, \frac{dx}{(x-a_1)(x-a_3)}, \ldots, \frac{dx}{(x-a_1)(x-a_g)}
\end{align}
form a basis for $H^0(C, \omega_C \otimes \scl^\vee)$.
The kernel of $\mu_0$ is then spanned by elements of the form
\begin{align*}
	(x-a_1) \otimes r_i + 1 \otimes t_i 
\end{align*}
for $1 \leq i \leq g-2$ with $r_i, t_i$ linear combinations of the elements in \eqref{equation:char-2-less-l}.

Therefore, $\im \mu_1$ is generated by $d(x-a_1) \otimes r_i + d(1) \otimes t_i  = dx \otimes r_i$ for $1 \leq i \leq g-2$.
In particular, the $r_i$ are some linear combination of elements appearing in \eqref{equation:char-2-less-l}.
It follows that $dx \otimes r_i$ are also linear combinations of the elements appearing in \eqref{equation:image-char-2-mult-basis}.
Therefore, $\im \mu_1 \subset \im (H^0(C, \omega_C^{\otimes 2}) \leftarrow \sym^2 H^0(C, \omega_C)).$
By \autoref{lemma:dual-tangent-space-description},
$\dim \im \mu_1 = g-2$ and $\dim \im (T_{[C]} \tau_g)^\vee = 2g-1$.
Hence,
the composition \eqref{equation:ag-to-hg-dual} has image of dimension $(2g-1) - (g-2) = g + 1$
and cokernel of dimension $(2g-1) - (g+1) = g-2$.
\end{proof}

\section{Proof of \autoref{theorem:hyperelliptic}}
\label{section:proof-hyperelliptic}

At this point, nearly everything is in place to prove \autoref{theorem:hyperelliptic}.
The work of previous sections will show that $\phi_g$ is a monomorphism over $\spec \bz[1/2]$, but fails to be a monomorphism when restricted to any field $k$ with $\chr(k) = 2$.
In order to show $\phi_g: \hg \ra \ag$ is an immersion over $\spec \bz[1/2]$, we will verify the valuative criterion for immersions (or radimmersions) in \autoref{proposition:valuative}. 
We then use \autoref{corollary:valuative} to deduce $\phi_g$ is an immersion over $\spec \bz[1/2]$.
This valuative criterion loosely says that a monomorphism $f: X \ra Y$ is an immersion
when, given a map from a the spectrum of a discrete valuation ring to $Y$ with its two points factoring through $X$, the map from the spectrum of the discrete valuation ring factors through $X$.

We next set up some notation.
Let $\overline{\mathscr M}_g$ denote the Deligne-Mumford compactification of $\mg$ 
and let
$\mgc$ denote the moduli stack of stable compact type curves of genus $g$. 
Recall that $\mgc$ can be constructed as the open substack of $\overline{\mathscr M}_g$
which, loosely speaking, parameterizes curves whose dual graph of components is a tree.
The geometric points of $\overline{\mathscr M}_g$ lie in $\mgc$ precisely when the Jacobian of the corresponding curve is an abelian variety, as follows from
\cite[\S 9.2, Example 8]{BoschLR:Neron}.
The following lemma is well-known:

\begin{lemma}
	\label{lemma:open-and-proper}
	The Torelli map $\phi_g: \hg \ra \ag$ factors as the composition of an immersion $\hg \ra \mgc$ and a proper map $\tau_g^c: \mgc \ra \ag$.	
\end{lemma}
\begin{proof}
	It is shown in \cite[Corollary 5.4]{alexeev:compactified-jacobians-and-the-torelli-map} that there is a compactification of
	$\ag$, which we denote $\overline{\mathscr A}_g$, and a map $\overline \tau_g : \overline{\mathscr M}_g \ra \overline{\mathscr A}_g$ extending $\tau_g$.
	(In \cite[Corollary 5.4]{alexeev:compactified-jacobians-and-the-torelli-map}, it is only stated that this yields a map of coarse spaces, but the map is
	in fact
	constructed as a map of stacks.)
Since $\overline{\mathscr M}_g$ and $\overline{\mathscr A}_g$ are proper, $\overline{\tau}_g$ is as well, and therefore the resulting map 
$\mgc = \overline{\mathscr M}_g \times_{\overline{\mathscr A}_g} \ag \ra \ag$ is also proper.
Finally, $\hg \ra \mgc$ is an immersion because it is a composition of the immersions $\hg \ra \mg \ra \mgc$.
\end{proof}
\begin{remark}
	\label{remark:}
	We used the rather difficult result of \cite[Corollary 5.4]{alexeev:compactified-jacobians-and-the-torelli-map} in the proof of \autoref{lemma:open-and-proper},
	but \autoref{lemma:open-and-proper} can also be verified directly. 
	One can extend the principal polarization on the universal Jacobian over $\mg$ to that over $\mgc$, and then use the valuative criterion of properness and
	\cite[\S7.4, Proposition 3]{BoschLR:Neron}
	to verify properness of $\mgc \ra \ag$.
	See 
	\cite{mathoverflow:does-the-compatified-torelli-map-extend}
	for further discussion of this.
	For the sake of brevity, we omit this more direct proof.
\end{remark}

We now verify $\phi_g$ satisfies the valuative criterion for radimmersions.
\begin{proposition}
	\label{proposition:valuative}
	The map $\phi_g: \hg \ra \ag$ satisfies the valuative criterion for radimmersions for traits as in \autoref{definition:valuative}.
\end{proposition}
\begin{proof}
	Let $T$ be a trait (the spectrum of a discrete valuation ring) with $2$-commutative diagrams as in \eqref{equation:closed-and-generic-point} for $X = \hg$ and $Y = \ag$.
	By \autoref{lemma:open-and-proper}
	and the valuative criterion for properness 
	\cite[Theorem 11.5.1]{olsson2016algebraic}
	there is a dominant map of traits $T' \ra T$
	so that the resulting map $T' \ra \ag$ factors through $\mgc$. 
	Let $\scj \ra T'$ denote the family of principally polarized abelian varieties of dimension $g$ corresponding to the map $T' \ra \ag$.
	The factorization $T' \ra \mgc$ yields a family of stable curves
	$\scc \ra T'$, so that the principally polarized Jacobian of the generic fiber of $\scc$ agrees with the generic fiber of $\scj \ra T'$.

	With notation for $T, T', t, t', \eta, \eta'$ as in \autoref{definition:valuative},
	we have diagrams
	\begin{equation}
		\label{equation:hg-to-ag}
		\begin{tikzcd} 
			\eta' \ar{r} \ar{dd} & \eta \ar {r} \ar {dd} & \hg \ar {d}  & t' \ar{r} \ar{dd} & t \ar {r} \ar {dd} & \hg \ar {d} \\
			&& \mgc  \ar{d}{\tau_g^c}  &  && \mgc\ar{d}{\tau_g^c} \\
			T' \ar{urr} \ar {r} & T \ar{r} & \ag &  T' \ar{urr}\ar {r} & T \ar{r} & \ag
	\end{tikzcd}\end{equation}
	The first diagram in \eqref{equation:hg-to-ag} $2$-commutes by the valuative criterion for properness.	
	We claim the second diagram also $2$-commutes. 
	Granting this claim, observe that $\hg \ra \mgc$ is an immersion, being the composition of a closed immersion $\hg \ra \mg$ and an open immersion $\mg \ra \mgc$.
	Therefore, the valuative criterion for immersions \autoref{corollary:valuative} implies $T' \ra \mgc$ lifts to a map $T' \ra \hg$ making the diagrams
	\eqref{equation:extension} $2$-commute for $X = \hg$ and $Y = \ag$.

	So, to conclude, we just need to check the second diagram in \eqref{equation:hg-to-ag} $2$-commutes.
	This will necessarily follow if we verify the fiber product $[C] \times_{\phi_g, \ag, \tau_g^c} \mgc$ contains a unique geometric point for any geometric point $[C] \in \hg$.
	First, we show $[C] \times_{\phi_g, \ag, \tau_g^c} \mgc$ does not contain any geometric points mapping to $\mgc - \mg$.
	Indeed, 
	the 
	theta divisor associated to a singular compact type
	curve is always geometrically reducible, while that associated to a curve in $\mg$ is geometrically irreducible.
	Hence, no geometric points of $\mgc - \mg$ can map under $\tau_g^c$ to $\phi_g([C])$.
	Therefore, it suffices to show $[C] \times_{\phi_g, \ag, \tau_g^c} \mgc$ contains a unique geometric point in $\mg$, which 
	follows from the classical Torelli theorem 
	\cite[Chapter VII, Theorem 12.1]{CornellS:Arithmetic}.
	\end{proof}

\begin{corollary}
	\label{corollary:phi-radimmersion}
	The map $\phi_g: \hg \ra \ag$ is a radimmersion (as defined in \autoref{definition:radimmersion}).
\end{corollary}
\begin{proof}
	This will follow from \autoref{proposition:valuative} and \autoref{theorem:valuative}
	once we verify $\phi_g$ is representable and induces a universal homeomorphism on isotropy groups at every point of $\hg$.
	We first check that $\phi_g$ induces isomorphisms on isotropy groups at every point of $\hg$.
	Observe that $\phi_g$ induces a bijection on geometric points of isotropy groups at every point of $\hg$ by the Torelli theorem \cite[Chapter VII, Theorem 12.1(b)]{CornellS:Arithmetic}.
	Since both $\hg$ and $\ag$ are Deligne-Mumford stacks, their isotropy groups at any geometric point are constant group schemes.
	Therefore, $\phi_g$ induces an isomorphism on isotropy groups at geometric points of $\hg$.
	As a consequence, $\phi_g$ is representable, as follows by applying \cite[Theorem 2.2.5(1)]{Conrad:arithmeticModuliOf} 
	to the pullback of
	$\phi_g$ along a schematic cover of $\ag$.
	Note that \cite[Theorem 2.2.5]{Conrad:arithmeticModuliOf} assumes stacks have finite type separated diagonals, but these assumptions apply in 
	this case as $\mg$ and $\ag$ even have finite diagonals.
\end{proof}

Combining \autoref{corollary:phi-radimmersion} with \autoref{lemma:dual-tangent-space-description}, \autoref{proposition:char-0-tangent-space}, and \autoref{proposition:char-2-tangent-space},
we now prove \autoref{theorem:hyperelliptic}.

\begin{proof}[Proof of \autoref{theorem:hyperelliptic}]
We know $\phi_g$ is a radimmersion by \autoref{corollary:phi-radimmersion}.
Next, we show $\phi_g$ is an immersion over $\spec \bz[1/2]$ or when $g =2$. 
Using \autoref{theorem:valuative} and \autoref{corollary:valuative},
we only need to check $\phi_g$ is a monomorphism over $\spec \bz[1/2]$ or when $g = 2$.
Equivalently, we just need to verify $\phi_g$ is injective on geometric points and tangent vectors \cite[17.2.6]{EGAIV.4}.
It follows from the classical Torelli theorem 
\cite[Chapter VII, Theorem 12.1(a)]{CornellS:Arithmetic}
that the map $\phi_g: \hg \ra \ag$ is injective
on geometric points.
Since smooth hyperelliptic curves $C$ in over a field $k$ of any characteristic can be deformed to smooth hyperelliptic curves of characteristic $0,$ (as can be seen by using explicit equations,) we obtain
an identification of
the kernel of $T_{[C]} \phi_g: T_{[C]}\hg \ra T_{[C]} \ag$ with $\ker \left(T_{[C]} \phi_{g,k} \right)$.
Note that $T_{[C]} \phi_{g,k}$ is given as the composition \eqref{equation:tangent-space-composition}.
The vanishing of $\ker T_{[C]} \phi_{g,k}$ 
over $\spec \bz[1/2]$ follows 
from combining \autoref{lemma:dual-tangent-space-description} with \autoref{proposition:char-0-tangent-space}.
Therefore, $\phi_g$ is injective on tangent vectors over $\spec\bz[1/2]$.
In the case $g = 2$ we find $\phi_g$ is a monomorphism by combining the above with \autoref{proposition:char-2-tangent-space}.

To conclude the proof, we just need to check that the restriction of $\phi_g$ to a field of characteristic $2$
induces a map on tangent spaces with kernel of dimension $g -2$.
Again using the identification $\ker T_{[C]} \phi_g =  \ker T_{[C]} \phi_{g,k}$ mentioned above, 
this follows from combining \autoref{lemma:dual-tangent-space-description} with \autoref{proposition:char-2-tangent-space}.
\end{proof}

\appendix
\section{The valuative criterion for locally closed immersions}
\label{section:valuative}

\subsection{Statement of the criterion}
In this section we state a valuative criterion for immersions of algebraic stacks in \autoref{corollary:valuative}.
The result is a generalization of \cite[Chapter 1, Corollary 2.13]{mochizuki2014foundations} used in proving \autoref{theorem:hyperelliptic}. 
Many of the ideas are present in \cite[Chapter 1, \S 2.4]{mochizuki2014foundations}, though nontrivial care
has to be taken to deal with algebraic stacks in place of schemes.

We begin by introducing definitions to state the valuative criterion for radimmersions, which will imply the analogous valuative criterion for immersions.
Recall that a morphism of schemes $X \ra Y$ is {\em radicial} if for every field $K$, $X(\spec K) \ra Y(\spec K)$ is injective, or equivalently each geometric fiber has at most one geometric point.
A morphism $f: X \ra Y$ of algebraic stacks is an immersion (or locally closed immersion) if it factors as a composition
$X \ra U \ra Y$ where $U \ra Y$ is an open immersion and $X \ra U$ is a closed immersion.
\begin{definition}
	\label{definition:radimmersion}
	A morphism of algebraic stacks $X \ra Y$ is a {\em radimmersion} if it factors as a composition
	$X \ra U \ra Y$ where $U$ is an algebraic stack with $U \ra Y$ an open immersion and
	$X \ra U$ a finite radicial map.
\end{definition}
We note in particular that radimmersions are representable by schemes.
\begin{remark}
	\label{remark:}
	In the context of maps of algebraic stacks, being radicial is not equivalent to being injective on geometric points as it is for maps of schemes.
	For example, for $k$ a field, $\spec k \ra \left[\spec k / \left( \bz/2\bz \right)  \right]$ is bijective on geometric points but is not radicial
because after pulling back to a schematic cover of the target, the resulting map is not radicial.
This distinction will play a significant role in what follows.
\end{remark}

For the next definition recall that a {\em trait} is a scheme of the form 
$\spec R$ for $R$ a discrete valuation ring.
\begin{definition}
	\label{definition:valuative}
Let $f: X \ra Y$ be a map of algebraic stacks.
We say {\em $f$ satisfies the valuative criterion for radimmersions} if
the following property holds:
Let $T$ be the spectrum of a discrete valuation ring
with generic point $\eta$ and closed point $t$,
and let $g: T \ra Y$ be any map.
Suppose we have $2$-commutative diagrams
	\begin{equation}
		\label{equation:closed-and-generic-point}
		\begin{tikzcd} 
			\eta \ar {r}{j_\eta} \ar {d}{i_\eta} & X \ar {d}{f} & t \ar{r}{j_t} \ar{d}{i_t} & X \ar{d}{f} \\
			T \ar {r}{g} & Y & T \ar{r}{g} & Y,
		\end{tikzcd}\end{equation}
with $2$-morphisms $\gamma_\eta: f\circ j_\eta \simeq g \circ i_\eta$ and $\gamma_t: f \circ j_t \simeq g \circ i_t$
witnessing $2$-commutativity of the diagrams.
Then, there exists a spectrum of a discrete valuation ring $T'$ with closed point $t'$ and generic point $\eta'$
with a specified dominant map $T' \ra T$ 
such that
there is a unique morphism $h:T' \ra X$ making the diagrams	
	\begin{equation}
		\label{equation:extension}
		\begin{tikzcd} 
		\eta' \ar{r} \ar{d} & \eta \ar {r} \ar {d} & X \ar {d}{f} & t' \ar{r} \ar{d}& t \ar{r} \ar{d} & X \ar{d}{f} \\
		T' \ar {r}\ar[crossing over]{urr}{h} & T \ar {r}{g} & Y & T' \ar{r}\ar[crossing over]{urr}{h} & T \ar{r}{g} & Y
		\end{tikzcd}\end{equation}
$2$-commute compatibly with the above choices of $\gamma_\eta$ and $\gamma_t$ (as for dotted arrows in \cite[\href{https://stacks.math.columbia.edu/tag/0CLA}{Tag 0CLA}]{stacks-project}).

We say {\em $f$ satisfies the valuative criterion for radimmersions with $T = T'$} if
for every spectrum of a discrete valuation ring $T$
and diagrams \eqref{equation:closed-and-generic-point},
there exists a map $h: T \ra X$ such that \eqref{equation:extension} holds with $T' = T$ and the map $T' \ra T$ being the identity map.

We say {\em $f$ satisfies the valuative criterion for radimmersions for traits} if $f$ satisfies the
valuative criterion for radimmersions for all traits $T$ and $T'$.

We say {\em $f$ satisfies the valuative criterion for radimmersions for traits with $T = T'$} if $f$ satisfies the
valuative criterion for radimmersions with $T = T'$ for all traits $T$.
\end{definition}

We can now state the valuative criterion for radimmersions.
We note that in the case that $X$ and $Y$ are finite type schemes over a noetherian base $S$,
the first two conditions of \autoref{theorem:valuative}
were shown to be equivalent in 
\cite[Chapter 1, Theorem 2.12]{mochizuki2014foundations}.
Recall that for $S$ a scheme and $x: S \ra X$ a point of an algebraic stack, the isotropy group at $x$ is by definition the algebraic space
$\isom_X(x,x)$.

\begin{theorem}[Valuative criterion for radimmersions]
	\label{theorem:valuative}
	Let $f:X \ra Y$ be a representable finite type quasi-separated morphism
	of algebraic stacks with $Y$ locally noetherian.
	Then, the following are equivalent:
	\begin{enumerate}
		\item $f$ is a radimmersion
		\item $f$ induces a universal homeomorphism on isotropy groups at every geometric point of $X$, and $f$ satisfies the valuative criterion for radimmersions for traits with $T=T'$
		\item $f$ induces a universal homeomorphism on isotropy groups at every geometric point of $X$, and $f$ satisfies the valuative criterion for radimmersions for traits. 
	\end{enumerate}
\end{theorem}
The proof is given at the end of this section.
Before giving the proof, we deduce the following valuative criterion for locally closed immersions, which 
generalizes \cite[Corollary 2.13]{mochizuki2014foundations}.

\begin{corollary}[Valuative criterion for locally closed immersions]
	\label{corollary:valuative}
	Let $f:X \ra Y$ be a finite type quasi-separated 
	monomorphism
	of algebraic stacks with $Y$ locally noetherian.
	Then, the following are equivalent:
	\begin{enumerate}
		\item $f$ is an immersion 
		\item $f$ satisfies the valuative criterion for radimmersions for traits with $T=T'$
		\item $f$ satisfies the valuative criterion for radimmersions for traits. 
	\end{enumerate}
\end{corollary}
\begin{proof}
	Recall by definition that a map is a monomorphism if it is representable (i.e., representable by algebraic spaces) and is fppf locally a monomorphism.
	Further, 
	by 
	\cite[\href{https://stacks.math.columbia.edu/tag/04ZZ}{Tag 04ZZ}]{stacks-project}
	monomorphisms must induce isomorphisms on isotropy groups at every point,
	and so in particular, the map on isotropy groups a universal homeomorphism.

	Therefore, using \autoref{theorem:valuative} it suffices to show that a monomorphism $f: X \ra Y$
	is an immersion if and only if it is a radimmersion.
	Certainly immersions are monomorphisms and radimmersions.
	So, we just need to check a 
	radimmersion which is a monomorphism is an immersion.
	Both immersions and radimmersions are representable by schemes by definition, 
	and so it suffices to check a radimmersion monomorphism of schemes
	is an immersion.
	Further, by factoring $f : X \ra Y$ as a composition of a finite morphism and an open
	immersion, it suffices to check a finite monomorphism is a closed immersion.
	This is shown in \cite[18.12.6]{EGAIV.4}.
\end{proof}

\subsection{Remarks and Examples}
We next make some comments on the valuative criterion for radimmersions and give some examples and non-examples.
\begin{remark}
	\label{remark:}
	One can similarly state and prove a version of the valuative criterion for radimmersions \autoref{theorem:valuative} and
	the valuative criterion for locally closed immersions \autoref{corollary:valuative},
	where one removes the noetherian hypotheses on $Y$ at the cost of assuming $f$ is finitely presented (instead of just of finite type)
	and working with all valuation rings (instead of just discrete valuation rings). 
	The proof is essentially the same, where one replaces the references to the noetherian valuative criteria for properness and separatedness
	for discrete valuation rings with references to valuative criteria for properness and separatedness for general valuation rings.
\end{remark}

\begin{example}
	\label{example:}
	As we have seen in \autoref{proposition:valuative}, the restricted Torelli map $\phi_g: \hg \ra \ag$ satisfies the valuative criterion for radimmersions.
	Another example of a map of algebraic stacks which can be seen to be an immersion using the valuative criterion
	is the map from the moduli stack of smooth plane curves of degree $d$ for $d \geq 4$ to $\mg$. 
	Here, the moduli
	stack of plane curves can be defined by taking the open in the Hilbert scheme of plane curves corresponding to smooth plane curves, and quotienting
	by the $\pgl_3$ action. 
	See \cite[Remark 5.4]{landesman-swaminathan-tao-xu:rational-families} for some more details.
	We note that we do not know how to see either of these maps are immersions without the valuative criterion.
\end{example}
\begin{example}
	\label{example:}
	Radimmersions of algebraic stacks do not always induce isomorphisms on isotropy groups.
	For example, for $k$ a field of characteristic $p$, the map $\spec k \ra \left[ \spec k/ \mu_p \right]$ is a radimmersion that does not
	induce an isomorphism on isotropy groups.
	More generally, we can replace $\mu_p$ with any group scheme with a single geometric point over $\spec k$ in the above example.
\end{example}
\begin{example}
	\label{example:non-radicial-example}
	We now give an example of a map which satisfies the valuative criterion for radimmersions for traits 
	but which is not a radimmersion.
	For $k$ a field, consider the representable map $f: \spec k \ra \left[ \spec k/ \left( \bz/2\bz \right) \right]$.
This satisfies the valuative criterion for radimmersions because a $\bz/2\bz$ torsor over a trait which is trivial over the generic fiber is necessarily
	trivial, using normality of the trait.
	Nevertheless, $f$ is not a radimmersion because the fiber of $f$ over $\spec k$ has two geometric points.
	In particular, $f$ induces the map $\left\{ \id \right\} \ra \bz/2\bz$ on isotropy groups, and so is not a universal homeomorphism on
	isotropy groups and hence does not satisfy \autoref{theorem:valuative}(3).
	More generally, one can replace $\bz/2\bz$ in the above example with any nontrivial constant group scheme.
\end{example}
\begin{example}
	\label{example:}
	In addition to \autoref{example:non-radicial-example},
	another example of a map which satisfies the valuative criterion for radimmersions for traits
	but which is not a radimmersion
	is the Torelli map $\tau_g: \mg \ra \ag$ when $g \geq 3$.
	This can be verified using the same method as in the proof of \autoref{proposition:valuative}.
	Of course, $\tau_g$ does not induce an isomorphism on isotropy groups because a generic genus $g$ curve for $g \geq 3$ has only the trivial
	automorphism, while all principally polarized abelian varieties have $\times[-1]$ as a nontrivial automorphism.
\end{example}
\begin{example}
	\label{example:}
	An example of a map which is bijective on geometric points but which does not satisfy the valuative criterion for radimmersions for traits is
	$\left[ \spec \br / \left( \bz/2\bz \right) \right] \ra \spec \br$, where $\br$ denotes the real numbers.
	This fails to satisfy the valuative criterion because one can map the generic point of a trait to a trivial $\bz/2\bz$ torsor and the 
	closed point to a nontrivial $\bz/2\bz$ torsor, and there will be no maps from the trait extending these.
	Indeed,
	any $\bz/2\bz$ torsor over over a trait which is generically trivial is trivial.
\end{example}

The above examples raise the following question:
\begin{question}
	\label{question:generalized-valuative}
	Is there a simple characterization of maps of algebraic stacks $f: X \ra Y$ which satisfy the valuative criterion for radimmersions?
\end{question}
Note that \autoref{question:generalized-valuative} is not answered by \autoref{theorem:valuative}
because we do not assume that $f$ is representable and 
a universal homeomorphism
on isotropy groups.

\subsection{Proving the valuative criterion}
Before proving \autoref{theorem:valuative} at the end of this section,
we establish a number of preliminary lemmas.
One of the main obstructions we face, not encountered in the schematic version from \cite[Chapter 1, \S2.4]{mochizuki2014foundations},
is to verify that $f$ is representable by schemes. This is verified using Zariski's main theorem in \autoref{lemma:schematic} after we show $f$
is separated.
We next verify that $f$ satisfying \autoref{theorem:valuative}(3) have geometric fibers with at most one geometric point.
\begin{lemma}
	\label{lemma:radicial}
	Suppose $f: X \ra Y$ is a finite type representable morphism of algebraic stacks, 
	inducing a universal homeomorphism on isotropy groups at each geometric point of $X$ and satisfying the valuative criterion for radimmersions for traits.
	Then each geometric fiber of $f$ has at most $1$ geometric point.
\end{lemma}
\begin{proof}
	Begin with a geometric point $y: \spec k \ra Y$.
	Suppose $\alpha: \spec k \ra X$ and $\beta: \spec k \ra X$ are two geometric points of $X$ with $2$-morphisms $f \circ \alpha \simeq y \simeq f \circ \beta$.
	Because the map induced by $f$ is a universal homeomorphism 
	on isotropy groups, if $\alpha$ and $\beta$ map to $2$-isomorphic points of $X$, they must map to the same point of $X_y := \spec k \times_{y, Y, f} X$.
	(In general, this property may fail when $f$ is not a universal homeomorphism
	on isotropy groups, such as in the case of $f: \spec k \ra \left[ \spec k/ \left( \bz/2\bz \right) \right]$.)
	Therefore, in order to show $X_y$ has at most one geometric point,
	it suffices to exhibit a $2$-morphism $\alpha \simeq \beta$.

	By the finite type hypothesis $\alpha$ and $\beta$ both factor though closed points of $X_y$.
	On the other hand, taking $T = \spec k\left[ \left[ x \right] \right]$
	in \autoref{definition:valuative}, we may choose diagrams \eqref{equation:closed-and-generic-point}
	sending the generic point of $T$ to the image of $\alpha$ via the inclusion $k \hookrightarrow k\left( \left( x \right) \right)$ and the closed point of $T$ to the image of $\beta$.
	Therefore, by the valuative criterion, we obtain a map $T' \ra X$ sending the generic point of $T'$
	to $\alpha$ and the closed point to $\beta$. Hence $\beta$ lies in the closure of $\alpha$.
	Because $\alpha$ and $\beta$ are both closed geometric points, we find $\beta \simeq \alpha$. 
	Therefore, $X$ has at most one geometric point over $y$ and hence $X_y$ has at most one geometric point by the
	preceding paragraph.
\end{proof}

\begin{lemma}
	\label{lemma:separated}
	Suppose $f: X \ra Y$ is a finite type quasi-separated representable 
	morphism of algebraic stacks with $Y$ locally noetherian,
	and suppose $f$ satisfies \autoref{theorem:valuative}(3).
	Then $f$ is separated.
\end{lemma}
\begin{proof}
By
\cite[\href{https://stacks.math.columbia.edu/tag/0E80}{Tag 0E80}]{stacks-project}
to show $f$ is separated, 
it suffices to verify the uniqueness part of the valuative criterion for discrete
	valuation rings.
	So, suppose we are given some dominant map of traits $T' \ra T$ and two maps $h_1: T' \ra X$ and $h_2 : T' \ra X$ (in place of $h$)
	making the first diagram in \eqref{equation:extension} $2$-commute.
We claim the second diagram in \eqref{equation:extension} also $2$-commutes.
First, observe that by \autoref{lemma:radicial}
	for any geometric point $y : \spec k \ra Y$, the fiber of $f$,
	$X_y := X \times_{Y,y} \spec k$, is $0$-dimensional and quasi-separated.
Therefore, $X_y$ is a scheme by \cite[Theorem 6.4.1]{olsson2016algebraic}.
So, to show the maps $\alpha_1: t' \ra T' \xra{h_1} X_y$ and $\alpha_2: t' \ra T' \xra{h_2} X_y$ agree, it suffices to show their images map to the same geometric point.
	This follows from \autoref{lemma:radicial}, because
	$f \circ \alpha_1\simeq f \circ \alpha_2$ by $2$-commutativity of the right diagram of \eqref{equation:extension}.
	Hence, it follows that both diagrams in \eqref{equation:extension} commute, and so $h_1$ agrees with $h_2$ by the uniqueness aspect of the valuative criterion for radimmersions for traits.
\end{proof}
We now deduce that morphisms of algebraic spaces satisfying the valuative criterion for radimmersions for traits are representable by schemes.
\begin{lemma}
	\label{lemma:schematic}
	Suppose $f: X \ra Y$ is a finite type quasi-separated morphism of algebraic spaces with $Y$ a locally noetherian scheme,
	and suppose $f$ satisfies the valuative criterion for radimmersions for traits.
	Then $X$ is in fact a scheme.
\end{lemma}
\begin{proof}
	By \autoref{lemma:radicial}, $f$ is radicial, hence quasi-finite. By \autoref{lemma:separated} $f$ is separated.
	Hence, by
	\cite[\href{https://stacks.math.columbia.edu/tag/082J}{Tag 082J}]{stacks-project},
	(a variant of Zariski's main theorem,)
	$f$ is quasi-affine, and therefore $f$ is a representable by schemes.
	Therefore, $X$ is a scheme.
\end{proof}

We next state a lemma with the goal of establishing $(1) \implies (2)$ in \autoref{theorem:valuative}.
\begin{lemma}
	\label{lemma:radimmersions-satisfy-valuative-spaces}
	If $f: X \ra Y$ is a radimmersion of schemes with $Y$ a locally noetherian scheme, then $f$ satisfies the valuative criterion for traits with $T=T'$.
\end{lemma}
\begin{proof}
	We can factor $f$ as a composition $X \ra U \ra Y$ for $X \ra U$ a finite radicial morphism and $U \ra Y$ an open immersion.
	Suppose we have commutative diagrams as in \eqref{equation:closed-and-generic-point}.
	Because both $t$ and $\eta$ factor through $U \subset Y$, we may replace $Y$ by $U$.
	Then, we may assume the map $f$ is finite, and in particular proper. By the valuative criterion for properness, a morphism $h: T' = T \ra Y$ exists making
	the first diagram in \eqref{equation:extension} commute.
	The second diagram in \eqref{equation:extension} then also commutes since the map $t \ra Y$ is uniquely determined by the composition $t \ra X \xra{f} Y$ 
	by \autoref{lemma:radicial}.
\end{proof}
We now bootstrap the preceding lemma to morphisms of algebraic stacks.
\begin{corollary}
	\label{corollary:radimmersions-satisfy-valuative-stacks}
	If $f: X \ra Y$ is a radimmersion of algebraic stacks with $Y$ locally noetherian, then $f$ satisfies the valuative criterion for radimmersions for traits with $T=T'$.
\end{corollary}
\begin{proof}
	Suppose we are given $2$-commuting diagrams as in \eqref{equation:closed-and-generic-point} with $T$ and $T'$ traits.
	Because $X \ra Y$ is a radimmersion, the fiber product $X \times_Y T$ is a scheme and the resulting map $X \times_Y T \ra T$ is a radimmersion.
	By the universal property of fiber products, we obtain $2$-commuting diagrams
	\begin{equation}
		\label{equation:lift-fiber}
		\begin{tikzcd} 
			\eta \ar {r} \ar {d} & X \times_Y T \ar{r} \ar {d} & X \ar {d} & t \ar{r} \ar{d} & X \times_Y T \ar{r} \ar{d} & X \ar{d}{f} \\
			T \ar {r}{\id} & T \ar{r} & Y & T \ar{r}{\id} & T \ar{r} & Y,
		\end{tikzcd}\end{equation}
		Hence, by applying \autoref{lemma:radimmersions-satisfy-valuative-spaces} to the left squares in the above diagrams, we obtain a unique lift
		$h': T \ra X \times_Y T$ making the squares in \eqref{equation:lift-fiber} $2$-commute.
		This implies there is a unique lift $h: T \ra X$ making \eqref{equation:extension} $2$-commute
		as in the valuative criterion for radimmersions for traits with $T = T'$.
		Specifically, $h$ is the composition of $h'$ with the projection $X \times_Y T \ra X$.
\end{proof}

For the implication $(1) \implies (2)$ we will also need the following verification that the map on isotropy groups is a universal homeomorphism.
\begin{lemma}
	\label{lemma:homeomorphism-isotropy}
	A radimmersion $f: X \ra Y$ of algebraic stacks induces a universal homeomorphism on isotropy groups at each geometric point of $X$.
	\end{lemma}
\begin{proof}
	We can factor $f: X \ra Y$ as $X \ra U \ra Y$ where $X \ra U$ is finite radicial and $U \ra Y$ is an open immersion.
	An open immersion induces an isomorphism on isotropy groups at every point of the source, so it suffices to prove the lemma in the case that $f$ is finite radicial.
	Choose a geometric point $x: \spec k \ra X$ and let $y: \spec k \ra X \ra Y$ denote the composition.
	Let $X_y := X \times_{f,Y,y} \spec k$. The map $x: \spec k \ra X$ and the $2$-morphism $x \circ f \simeq y$ induce a map
	$z : \spec k \ra X_y$.
	Then, we have the following diagram, where all squares are cartesian
	\begin{equation}
		\label{equation:}
		\begin{tikzcd} 
			\isom_X(x,x) \ar {r} \ar {d} &\spec k \ar {d}{z}  & \\
			\isom_Y(y,y) \ar {r}{h} \ar {d}& X_y \ar {r}{g} \ar {d} & \spec k \ar {d}{y} \\
			\spec k \ar {r}{x} & X \ar {r}{f} & Y.
	\end{tikzcd}\end{equation}
	Because $f$ is a finite radicial map, $g: X_y \ra \spec k$ is also finite radicial.
	Additionally, $g$ is surjective because $z$ is a section of $g$.
	So, $g$ is finite radicial and surjective, therefore a universal homeomorphism \cite[18.12.11]{EGAIV.4}.
	Since $z: \spec k \ra X_y$ is a section of $g$, it is also a universal homeomorphism.
	Therefore, the map $\isom_X(x,x) \ra \isom_Y(y,y)$,
	which is the base change of $z$ along $h$, is also a universal homeomorphism.
\end{proof}

The following lemma will be useful for reducing the implication $(3) \implies (1)$ to the case that $Y$ is an algebraic space.
\begin{lemma}
	\label{lemma:local-valuative}
	Suppose $\alpha : \scx \ra \scy$ is a finite type quasi-separated representable morphism of algebraic stacks with $Y$ locally noetherian
	and suppose $\alpha$ satisfies \autoref{theorem:valuative}(3).
	Then, for any scheme $Z$, the base change map $\scx \times_\scy Z \ra Z$ also satisfies the valuative criterion for radimmersions for traits.
\end{lemma}
\begin{proof}
	Given any map $T \ra Z$ from a trait $T$, making \eqref{equation:closed-and-generic-point} commute with $X = \scx \times_\scy Z$ and $Y = Z$,
	we wish to show there is a unique dominant map of traits $T' \ra T$ and $h: T' \ra \scx \times_\scy Z$ making the resulting diagrams
	in \eqref{equation:extension} commute.
	Since $\scx \ra \scy$ satisfies the valuative criterion for radimmersions for traits, we obtain a map $T' \ra \scx$ making
	the diagrams in \eqref{equation:extension} associated to the map $\alpha: \scx \ra \scy$ $2$-commute.
	Since we are also given a map $T' \ra Z$, we obtain the desired map $h: T' \ra \scx \times_\scy Z$ making
	\eqref{equation:extension} commute.
	We only need verify uniqueness of the map $h$.
	There is a unique geometric point $q$ of $\scx \times_\scy Z$ over the image of $\eta'$ in $Z$ by \autoref{lemma:radicial}.
	This implies that any such map $h$ must send $\eta' \mapsto q$.
	From the valuative criterion for separatedness, to show the map $h$ is unique, it suffices to show
	$\scx \times_\scy Z \ra Z$ is separated.
	This holds by \autoref{lemma:separated} because $\scx \times_\scy Z \ra Z$ is the base change of the separated map $\scx \ra \scy$.
\end{proof}

We will also need a lemma which allows us to check that a morphism is a radimmersion fppf locally.
\begin{lemma}
	\label{lemma:radimmersion-local}
	Suppose $f : X \to Y$ is map of schemes and $g:Z \to Y$ is an fppf map of schemes.
	If $X \times_Y Z \to Z$ is a radimmersion then so is $f$.
\end{lemma}
\begin{proof}
	Let $X \times_Y Z \to U' \to Z$ be a factorization of $X \times_Y Z \to Z$ as the composition of a finite radicial map
	and an open immersion.
	Let $U \subset Y$ denote the image of $U'$, which is open because $g$ is flat and locally of finite presentation.
	Note that $f: X \to Y$ necessarily factors through $U$ because $X \times_Y Z \to Z$ factors through $Z$ and $g$ is fppf.
	To conclude, it suffices to check $X \to U$ is finite radicial because the properties of being finite and radicial
	can be checked fppf locally on the target. 
	Since $U' \to U$ is fppf, it is enough to check $X \times_U U' \to U'$ is finite radicial.
	However, there is an isomorphism $X \times_U U' \simeq X \times_Y Z$ identifying the map $X \times_U U' \to U'$ with the map $X \times_Y Z \to U'$,
	and so $X \times_U U' \to U'$ is indeed finite radicial.
\end{proof}

The following lemma will let us deal with the important special case that $Y$ is a strictly henselian scheme.
\begin{lemma}
	\label{lemma:strictly-henselian-case}
	Suppose $f: X \to Y$ is a finite type quasi-separated map of schemes with $Y$ a strictly henselian Noetherian local scheme whose closed point lies in the image of $f$.
	If $f$ satisfies the valuative criterion for radimmersions for traits then $f$ is finite radicial.
\end{lemma}
\begin{proof}
	From \autoref{lemma:radicial}, $f$ is radicial, hence quasi-finite. From \autoref{lemma:separated}
	$f$ is separated.
	So, by Zariski's main theorem \cite[18.12.13]{EGAIV.4}, we find $f$ can be factored as $X \ra Z \ra Y$ with $\alpha: X \ra Z$ an open immersion and $\beta: Z \ra Y$ finite.
	By 
	\cite[18.5.11(a)]{EGAIV.4},
	$Z$ is a disjoint union $C_1 \coprod \cdots \coprod C_r$ with $C_i = \spec R_i$, for $R_i$ a local ring.
	We next show that, after possibly relabeling the $C_i$, $\alpha$ factors through $C_1$.
	Chose a point $x \in X$ with $f(x)$ mapping to the closed point of $Y$.
	We may assume $x \in C_1$. Choose some $w \in X$ lying in $C_i$ for some $i, 1 \leq i \leq r$. We will show that $i = 1$.
	Since
$Y$ is noetherian, 
	we may construct a trait $T$ whose closed point maps to $f(x)$ and whose generic point maps to $f(w)$
	\cite[\href{https://stacks.math.columbia.edu/tag/054F}{Tag 054F}]{stacks-project}.
	Hence, by the valuative criterion for radimmersions for traits, 
	we can find an extension $T' \ra T$ of traits whose closed point maps to $x$ and whose generic point maps to $w$.
	Therefore, $x$ is in the closure of $w$ and $i = 1$.
	We conclude that we $\alpha$ factors through $C_1$, so we may take $\alpha$ to be an isomorphism, as $\im \alpha$ contains the closed point of $C_1$.
	Hence $f$ is finite radicial.
\end{proof}

Combining the above, we now prove \autoref{theorem:valuative}.
\begin{proof}[Proof of \autoref{theorem:valuative}]
	(1) implies (2) by \autoref{corollary:radimmersions-satisfy-valuative-stacks} and \autoref{lemma:homeomorphism-isotropy}. 
	Also,
	(2) implies (3) by definition. 
	So it remains to check (3) implies (1).

	In order to verify $f: X \ra Y$ is a radimmersion, we may do so smooth locally on $Y$ by \autoref{lemma:radimmersion-local}.
	Therefore, 
	for $U \ra Y$ a smooth cover, the map $U \times_Y X \ra U$ again satisfies (3) by \autoref{lemma:local-valuative}.
	Hence, we may assume that $Y$ is a scheme.
	By representability of $f$, $X$ is an algebraic space.
	By \autoref{lemma:schematic}, we find that $X$ is also a scheme.
	Since $Y$ is locally noetherian, and $f$ is finite type, $f$ is in fact finitely presented.

	For any $x \in X$, let $S_x$ denote the spectrum of the strict henselization of $Y$ at $f(x)$.
	By \autoref{lemma:strictly-henselian-case}, $X \times_Y S_x \ra S_x$ is a finite radicial map for all $x \in X$. We want to show this implies $f: X \to Y$ is radimmersion.
	Note that $S_x$ is again noetherian by \cite[18.8.8(iv)]{EGAIV.4}.
	Then, since $f$ is finitely presented, and $S_x$ can be expressed as a limit of finite \'etale covers of the local scheme of $Y$ at $f(x)$,
	it follows
	from spreading out for finite morphisms \cite[8.10.5(x)]{EGAIV.3}
	that there is an \'etale neighborhood $S_x'$ of $f(x)$ such that $X \times_Y S_x' \to S_x'$ is finite.
	Define $U := \cup_{x \in X} \im (S_x \to Y)$.
	By fppf descent for finite morphisms 
	\cite[\href{https://stacks.math.columbia.edu/tag/02LA}{Tag 02LA}]{stacks-project}
	it follows that $f: X \ra U$ is finite. 
	Further, $X \ra U$ is radicial because it is so on all fibers over points of $U$.
	Therefore, we find that $f$ factors through $U$, with $U\ra Y$ an open immersion and $X \ra U$ finite radicial.
	Hence, $f$ is a radimmersion.
	\end{proof}

\bibliographystyle{alpha}
\bibliography{/home/aaron/Dropbox/master}

\def\cprime{$'$} \providecommand{\noopsort}[1]{}
\begin{thebibliography}{ACGH85}

\bibitem[ACG11]{arbarelloCG:geomtry-of-algebraic-curves-ii}
Enrico Arbarello, Maurizio Cornalba, and Pillip~A. Griffiths.
\newblock {\em Geometry of algebraic curves. {V}olume {II}}, volume 268 of {\em
  Grundlehren der Mathematischen Wissenschaften [Fundamental Principles of
  Mathematical Sciences]}.
\newblock Springer, Heidelberg, 2011.
\newblock With a contribution by Joseph Daniel Harris.

\bibitem[ACGH85]{ACGH:I}
E.~Arbarello, M.~Cornalba, P.~A. Griffiths, and J.~Harris.
\newblock {\em Geometry of algebraic curves. {V}ol. {I}}, volume 267 of {\em
  Grundlehren der Mathematischen Wissenschaften [Fundamental Principles of
  Mathematical Sciences]}.
\newblock Springer-Verlag, New York, 1985.

\bibitem[Ale04]{alexeev:compactified-jacobians-and-the-torelli-map}
Valery Alexeev.
\newblock Compactified {J}acobians and {T}orelli map.
\newblock {\em Publ. Res. Inst. Math. Sci.}, 40(4):1241--1265, 2004.

\bibitem[And58]{andreotti:on-a-theorem-of-torelli}
Aldo Andreotti.
\newblock On a theorem of {T}orelli.
\newblock {\em Amer. J. Math.}, 80:801--828, 1958.

\bibitem[Bea96]{beauville:complex-algebraic-surfaces}
Arnaud Beauville.
\newblock {\em Complex algebraic surfaces}, volume~34 of {\em London
  Mathematical Society Student Texts}.
\newblock Cambridge University Press, Cambridge, second edition, 1996.
\newblock Translated from the 1978 French original by R. Barlow, with
  assistance from N. I. Shepherd-Barron and M. Reid.

\bibitem[BLR90]{BoschLR:Neron}
Siegfried Bosch, Werner L{\"u}tkebohmert, and Michel Raynaud.
\newblock {\em N\'eron models}, volume~21 of {\em Ergebnisse der Mathematik und
  ihrer Grenzgebiete (3) [Results in Mathematics and Related Areas (3)]}.
\newblock Springer-Verlag, Berlin, 1990.

\bibitem[Con07]{Conrad:arithmeticModuliOf}
Brian Conrad.
\newblock Arithmetic moduli of generalized elliptic curves.
\newblock {\em J. Inst. Math. Jussieu}, 6(2):209--278, 2007.

\bibitem[Cor86]{CornellS:Arithmetic}
Arithmetic geometry.
\newblock pages xvi+353, 1986.
\newblock Papers from the conference held at the University of Connecticut,
  Storrs, Connecticut, July 30--August 10, 1984.

\bibitem[EP13]{elkinP:ekedahl-oort-strata-of-hyperelliptic-curves-in-characteristic-2}
Arsen Elkin and Rachel Pries.
\newblock Ekedahl-{O}ort strata of hyperelliptic curves in characteristic 2.
\newblock {\em Algebra Number Theory}, 7(3):507--532, 2013.

\bibitem[FKW19]{farbKW:modular-functions-and-resolvent-problems}
Benson Farb, Mark Kisin, and Jesse Wolfson.
\newblock Modular functions and resolvent problems.
\newblock {\em arXiv preprint arXiv:1912.12536v1}, 2019.

\bibitem[Gro66]{EGAIV.3}
A.~Grothendieck.
\newblock \'{E}l\'ements de g\'eom\'etrie alg\'ebrique. {IV}. \'{E}tude locale
  des sch\'emas et des morphismes de sch\'emas. {III}.
\newblock {\em Inst. Hautes \'Etudes Sci. Publ. Math.}, (28):255, 1966.

\bibitem[Gro67]{EGAIV.4}
A.~Grothendieck.
\newblock \'{E}l\'ements de g\'eom\'etrie alg\'ebrique. {IV}. \'{E}tude locale
  des sch\'emas et des morphismes de sch\'emas {IV}.
\newblock {\em Inst. Hautes \'Etudes Sci. Publ. Math.}, (32):361, 1967.

\bibitem[Lan]{mathoverflow:does-the-compatified-torelli-map-extend}
Aaron Landesman.
\newblock Does the compactified torelli map extend to a proper map of stacks?
\newblock MathOverflow.
\newblock URL:https://mathoverflow.net/q/337964 (version: 2019-08-09).

\bibitem[Lan19]{landesman:the-infinitesimal-torelli-problem}
Aaron Landesman.
\newblock {The infinitesimal Torelli problem}.
\newblock {\em arXiv preprint arXiv:1911.02187v1}, 2019.

\bibitem[LSTX19]{landesman-swaminathan-tao-xu:rational-families}
Aaron Landesman, Ashvin Swaminathan, James Tao, and Yujie Xu.
\newblock Surjectivity of {G}alois representations in rational families of
  abelian varieties.
\newblock {\em Algebra \& Number Theory}, 13(5):995--1038, 2019.
\newblock With an appendix by Davide Lombardo.

\bibitem[Moc14]{mochizuki2014foundations}
Shinichi Mochizuki.
\newblock {\em Foundations of p-adic Teichm{\"u}ller Theory}, volume~11.
\newblock American Mathematical Soc., 2014.

\bibitem[Ols16]{olsson2016algebraic}
Martin Olsson.
\newblock {\em Algebraic spaces and stacks}, volume~62 of {\em American
  Mathematical Society Colloquium Publications}.
\newblock American Mathematical Society, Providence, RI, 2016.

\bibitem[OS79]{oortS:local-torelli-problem}
Frans Oort and Joseph Steenbrink.
\newblock The local {{Torelli}} problem for algebraic curves.
\newblock {\em Journ{\'e}es de G{\'e}ometrie Alg{\'e}brique d'Angers},
  1979:157--204, 1979.

\bibitem[Ric20]{ricolfi:hilbert-scheme-hyperelliptic}
Andrea~T. Ricolfi.
\newblock {The Hilbert scheme of hyperelliptic Jacobians and moduli of Picard
  sheaves}.
\newblock {\em Algebra Number Theory}, 14(6):1381--1397, 2020.

\bibitem[SD73]{saint-donat:on-petris-analysis-of-the-linear-system-of-quadrics}
B.~Saint-Donat.
\newblock On {P}etri's analysis of the linear system of quadrics through a
  canonical curve.
\newblock {\em Math. Ann.}, 206:157--175, 1973.

\bibitem[Ser06]{sernesi:deformations-of-algebraic-schemes}
Edoardo Sernesi.
\newblock {\em Deformations of algebraic schemes}, volume 334 of {\em
  Grundlehren der Mathematischen Wissenschaften [Fundamental Principles of
  Mathematical Sciences]}.
\newblock Springer-Verlag, Berlin, 2006.

\bibitem[{Sta}]{stacks-project}
The {Stacks Project Authors}.
\newblock {\itshape Stacks Project}.
\newblock \url{http://stacks.math.columbia.edu}.

\bibitem[Tor13]{torelli:sulle-varieta-di-jacobi}
R~Torelli.
\newblock Sulle varieta di jacobi, rendiconti ara d.
\newblock {\em Lincei}, 22(1913):98, 1913.

\end{thebibliography}

\end{document}